\documentclass[a4paper,12pt]{article}
\usepackage[utf8]{inputenc}
\usepackage{mathtools}
\usepackage{amsmath,amssymb,amsthm,amsfonts}
\usepackage{euscript,graphicx}
\usepackage{longtable}
\usepackage{url}
\usepackage{subfigure}
\usepackage{paralist}
\usepackage{tikz}
\usetikzlibrary{calc}
\usetikzlibrary{decorations.pathmorphing}
\usetikzlibrary{arrows}
\usetikzlibrary{decorations.markings}
\usetikzlibrary{patterns}
\usetikzlibrary{intersections}

\usepackage[text={16cm,24cm},centering]{geometry}

\newtheorem{theorem}{Theorem}
\newtheorem{proposition}[theorem]{Proposition}
\newtheorem{lemma}[theorem]{Lemma}
\newtheorem{corollary}[theorem]{Corollary}

\newtheorem{definition}[theorem]{Definition}

\theoremstyle{definition}
\newtheorem{example}[theorem]{Example}

\DeclareMathOperator{\Aut}{Aut}
\DeclareMathOperator{\girth}{girth}

\title{Coronoids, Patches and Generalised Altans}
\author{ 
Nino Ba\v si\' c
\footnote{\textit{Faculty of Mathematics and Physics, University of Ljubljana, Slovenia.}
E-mail: \texttt{nino.basic@fmf.uni-lj.si}}
\and
Patrick W.\ Fowler
\footnote{\textit{Department of Chemistry, University of Sheffield, Sheffield, S3 7HF, UK.}
E-mail: \texttt{p.w.fowler@sheffield.ac.uk}}
\and
Toma\v z Pisanski
\footnote{\textit{FAMNIT, University of Primorska, Slovenia.}
E-mail: \texttt{tomaz.pisanski@upr.si}}}

\date{September 22, 2015}

\begin{document}

\nocite{*} 

\maketitle

\begin{abstract}
In this paper we revisit coronoids, in particular multiple coronoids. We consider
a mathematical formalisation of the theory of coronoid hydrocarbons that is solely
based on incidence between hexagons of the infinite hexagonal grid in the plane.
In parallel, we consider perforated patches, which generalise coronoids: in addition
to hexagons, other polygons may also be present. Just as coronoids may be considered
as benzenoids with holes, perforated patches are patches with holes. Both cases,
coronoids and perforated patches, admit a generalisation of the altan operation that
can be performed at several holes simultaneously. A formula for the number of
Kekul\'{e} structures of a generalised altan can be derived easily if the number of
Kekul\'{e} structures is known for the original graph. Pauling Bond Orders for
generalised altans are also easy to derive from those of the original graph.
\end{abstract}


\vspace{0.5\baselineskip}
\noindent \textbf{Math.\ Subj.\ Class.\ (2010):}
92E10, 
05C10, 
05C90 

\section{Introduction}
The term `altan' was recently coined \cite{monaco2012} to describe a particular type of
conjugated $\pi$-system, defined by a notional expansion of the annulene-like perimeter
of a parent hydrocarbon. Mathematical formalisation \cite{gutman2014a,gutman2014} gives
an operation that can be applied to \emph{any} planar graph to produce the altan of the parent graph,
and to predict consequent changes in various properties of mathematical/chemical interest.
In a recent paper \cite{basic2015}, basic properties of iterated altans were studied.
For previous work on altans from both chemical and mathematical perspectives, the reader is referred to
\cite{basic2015,dickensmallion2014a,dickensmallion2014,gutman2014a,gutman2014,monacomemoli2013,monacozanasi2009,monaco2013}.
In the present paper we apply successive altan operations, not to a single \emph{perimeter} 
(or \emph{peripheral root}, as it was called in \cite{basic2015}), but to a collection of
disjoint perimeters. In particular, a composite operation of this type applies well to general
coronoids \cite{cyvin1991,gutman1989}, which, unlike benzenoids, may possess more than
one perimeter. In addition to the outer perimeter they have a perimeter for each of the holes.  
We call this operation a \emph{generalised altan}. Owing to its generality it applies to
\emph{single} coronoids, i.e.,\ to coronoids possessing exactly one corona hole, and
\emph{multiple} coronoids, i.e.,\ coronoids possessing more than one corona hole. 
It seems that in the past investigations of coronoids mostly \emph{single coronoids} were
considered \cite{cyvin1991,gutman1989}.

Finally, we consider Kekul\'{e} structures \cite{cyvin1988} of generalised altans. It turns
out to determine the number of Kekul\'{e} structures of a generalised altan, if the number
of Kekul\'{e} structures are known for the original graph.

In a previous paper we shifted attention from benzenoids to the more general subcubic
planar graphs that we called \emph{patches}, which generalise the \emph{fullerene patches} of Graver et al.\ 
\cite{graver2003,gravergraves2010,gravergravesgraves2014,gravesgraves2014,graves2012,extendingpatches2015}.
In this paper, we similarly generalise coronoids to \emph{perforated patches}, i.e.,\ to
patches with several disjoint holes. 

\section{Hexagonal Systems, Coronoids and Benzenoids}
Traditionally, a \emph{benzenoid} is a collection of hexagons that constitute a simply connected
bounded region of the infinite hexagonal grid $\mathcal{H}$ in the Euclidean plane. Other equivalent
definitions are also possible. For a complete treatment of this topic, see \cite{gutman1989}.
Many authors consider benzenoids as plane graphs. Take a benzenoid. Note that one of its faces, called
the \emph{outer face}, is unbounded. Moreover, every edge is incident to exactly two distinct faces
(of which one may be the outer face). Every vertex is incident to 2 or 3 edges. Therefore, a benzenoid
graph is 2-connected and 2-edge-connected. In a chemical context, we use terms \emph{atom} and \emph{bond}
as synonyms for respective vertex and edge. Intuitively speaking, a \emph{coronoid} is a ``benzenoid with holes'', 
i.e.,\ a benzenoid with some internal bonds and atoms removed.
To precisely define the class of coronoids, some additional restrictions are needed. Normally, the resulting
structure must be composed entirely of hexagons and connected if it is to be of interest in the theory
of conjugated carbon frameworks.
\begin{figure}[!h]
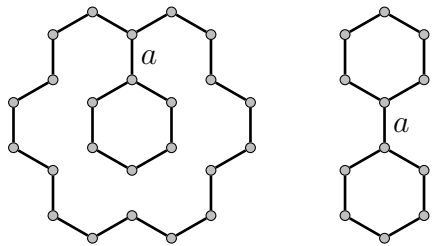

\centering
\input{img/figure_1a.tikz}
\qquad
\input{img/figure_1d.tikz}
\caption{These two plane graphs are not coronoids. The edge denoted by $a$ (in both examples) is not incident to two distinct faces.}
\label{fig:fig1}
\end{figure}
Although the two plane graphs on Figure \ref{fig:fig1} can be obtained from benzenoids by removing vertices
and edges, they are not coronoids. This motivates the following mathematical formalisation of coronoids and benzenoids.

In this paper, we will consider the infinite planar hexagonal grid $\mathcal{H}$ as a collection of hexagons.
Let $a$ and $b$ be two hexagons from $\mathcal{H}$. We say that $a$ and $b$ are \emph{adjacent}, and denote
this by $a \sim b$, if and only if they are different and share an edge. If two distinct hexagons share an edge
we call them \emph{neighbours}. The set of all neighbours of $a$ will be denoted $N(a)$. Note that each hexagon
of $\mathcal{H}$ has exactly 6 neighbours. Let $\mathcal{K} \subseteq \mathcal{H}$ be a \emph{hexagonal system},
i.e.,\ an arbitrary collection of hexagons from $\mathcal{H}$. $N_{\mathcal{K}}(a)$ will denote the set of those
hexagons of $\mathcal{K}$ that belong to $N(a)$, i.e.,\ $N_{\mathcal{K}}(a) \coloneqq N(a) \cap \mathcal{K}$. We
call them the \emph{neighbours of $a$ in $\mathcal{K}$}. Define equivalence relation $\equiv_{\mathcal{K}}$ as
follows: for $a, b \in \mathcal{K}$ it holds that $a \equiv_{\mathcal{K}} b$ if there exists a sequence
$c_0 = a, c_1, c_2, \ldots, c_m = b$ such that $c_{i-1} \sim c_i$ for $i = 1, 2, \ldots, m$ and
$c_i \in \mathcal{K}$ for $i = 0, 1, \ldots, m$. In particular, this means that it is possible to move from
hexagon $a$ to hexagon $b$ along a pathway composed of adjacent hexagons that all belong to $\mathcal{K}$
(see Figure \ref{fig:fig2}). Note that $a \equiv_\mathcal{K} b$ if $a \in N_{\mathcal{K}}(b)$ and $b \in \mathcal{K}$.
Also, it is easy to see that $\mathcal{K} \subseteq \mathcal{L}$ and $a \equiv_{\mathcal{K}} b$ imply $a \equiv_{\mathcal{L}} b$.
\begin{figure}[!h]
\centering
\input{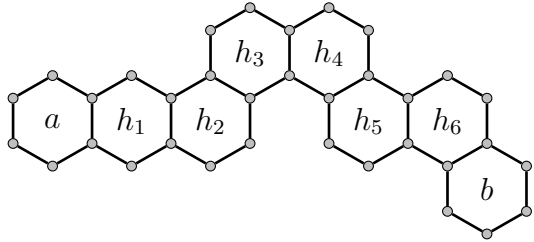}
\caption{A path of hexagons $h_1, h_2, \ldots, h_6$ joining hexagon $a$ to hexagon $b$.}
\label{fig:fig2}
\end{figure}
Hexagonal system $\mathcal{K}$ is naturally decomposed into equivalence classes $\{\mathcal{C}_i\}_{i \in C(\mathcal{K})}$,
called \emph{connected components}. Of course, $\mathcal{K} = \cup_{i \in C(\mathcal{K})} \mathcal{C}_i$. If $\mathcal{K}$
is finite, the number of its connected components, i.e.,\ the cardinality of $C(\mathcal{K})$, is also finite. A hexagonal
system is \emph{connected} if it comprises only one connected component.

\begin{lemma}
Let $\mathcal{K}$ be a hexagonal system and $\{\mathcal{C}_i\}_{i \in C(\mathcal{K})}$
its decomposition into equivalence classes. Let $\mathcal{L}$ be a connected hexagonal
system. If $\mathcal{L} \subseteq \mathcal{K}$ then $\mathcal{L} \subseteq \mathcal{C}_i$
for some $i \in C(\mathcal{K})$.
\end{lemma}

\begin{proof}
Suppose there exist hexagons $h_i \in \mathcal{L} \cap \mathcal{C}_i$ and
$h_j \in \mathcal{L} \cap \mathcal{C}_j$. Connectedness of $\mathcal{L}$ implies
$h_i \equiv_{\mathcal{L}} h_j$. From $\mathcal{L} \subseteq \mathcal{K}$ it follows
that $h_i \equiv_{\mathcal{K}} h_j$. A contradiction.
\end{proof}

\begin{lemma}
\label{lemma:neighbours}
Let $\mathcal{K} \neq \emptyset$ be an arbitrary hexagonal system and $a \in \mathcal{K}$
any of its hexagons. Then hexagon $a$ belongs to some connected component $\mathcal{C}_i$
of $\mathcal{K}$. Let $b \in N(a)$ be any of the neighbours of $a$ in $\mathcal{H}$. Then
either $b \in \mathcal{C}_i$ or $b \in \mathcal{K}^\complement$. In other words, no hexagon
of $\mathcal{C}_i$ is adjacent to a hexagon of $\mathcal{C}_j$ if $i \neq j$.
\end{lemma}

\begin{proof}
Suppose there is a hexagon $b \in \mathcal{C}_j$, $i \neq j$, such that $b \in N(a)$.
Then $a$ and $b$ are in the same equivalence class by the definition of $\equiv_\mathcal{K}$.
A contradiction.
\end{proof}

\begin{lemma}
\label{lemma:adjacencylemma}
Let $\mathcal{K}$ be an arbitrary non-empty hexagonal system that is a proper subset of
$\mathcal{H}$, i.e.,\ $\emptyset \neq \mathcal{K} \subset \mathcal{H}$. Let $\mathcal{C}_a$
be any connected component of the complement $\mathcal{K}^\complement$. Then there exists
a hexagon $\tilde{a} \in \mathcal{C}_a$ that is adjacent to some hexagon in $\mathcal{K}$.
\end{lemma}

\begin{proof}
Let $N(\mathcal{R})$ denote the set of all hexagons that are not contained in
$\mathcal{R}$ and are adjacent to some member of $\mathcal{R}$, i.e.,\ 
$N(\mathcal{R}) = \left( \cup_{r \in \mathcal{R}} N(r) \right) \setminus \mathcal{R}$.
Let $a \in \mathcal{C}_a$ and let $\mathcal{P}_0 = \{a\}$. Define
$\mathcal{P}_k = \mathcal{P}_{k-1} \cup N(\mathcal{P}_{k-1})$ for $k \geq 1$.

It is clear that $\mathcal{P}_0 \subseteq \mathcal{C}_a$. Let $n$ be the smallest number
such that $\mathcal{P}_n \nsubseteq  \mathcal{C}_a$. Suppose that such a number $n$ exists.
Note that $\mathcal{P}_{n-1} \subseteq \mathcal{C}_a$. Then $\mathcal{P}_n$ contains some
hexagon $h \notin \mathcal{C}_a$. By Lemma~\ref{lemma:neighbours}, $h \in \mathcal{K}$.
By construction of family $\{\mathcal{P}_k\}_{k}$, there exists a hexagon
$\tilde{a} \in \mathcal{P}_{n-1}$, such that $h \in N(\tilde{a})$.

Now suppose that the desired number $n$ does not exist. This means that
$\mathcal{P}_{n} \subseteq \mathcal{C}_a$ for all $n$. But
$\cup_{n=0}^{\infty} \mathcal{P}_n = \mathcal{H}$, i.e.,\ for every $h \in \mathcal{H}$
there exists some number $m$, s.t.\ $h \in \mathcal{P}_m$. Since $\mathcal{K}$ is non-empty
there is some hexagon $k \in \mathcal{K} \subset \mathcal{H}$ and therefore
$k \in \mathcal{P}_l$ for some number $l$, which is a contradiction.
\end{proof}

\begin{definition}
\label{def:coronoid}
A finite connected hexagonal system $\mathcal{K}$ is called a \emph{(general) coronoid}.
\end{definition}

\begin{definition}
\label{def:benzenoid}
A finite connected hexagonal system $\mathcal{K}$ whose complement $\mathcal{K}^\complement$
is also connected is called a \emph{benzenoid}.
\end{definition}

\noindent We prove a useful lemma:

\begin{lemma}
\label{lemma:decomposition}
Let $\mathcal{K}$ be any finite hexagonal system. Then the complement $\mathcal{K}^\complement$
of $\mathcal{K}$ consists of finitely many, say $d+1$, $d \geq 0$, connected components:
$$
\mathcal{K}^\complement = \mathcal{C}_\infty \sqcup \mathcal{C}_1 \sqcup \mathcal{C}_2 \sqcup \cdots \sqcup \mathcal{C}_d.
$$
All of these components but one, denoted by $\mathcal{C}_\infty$, is finite. Each of the
finite components $\mathcal{C}_i$, $1 \leq i \leq d$, is a coronoid.
\end{lemma}
\noindent Component $\mathcal{C}_\infty$ is called the \emph{exterior} of $\mathcal{K}$ and
each $\mathcal{C}_i$, $1 \leq i \leq d$, is called a \emph{corona hole}. In the above
expression $\sqcup$ stands for disjoint union. The size of a coronoid $\mathcal{K}$, denoted
$|\mathcal{K}|$, is the cardinality of the set $\mathcal{K}$, i.e.,\ the number of hexagons
it consists.
\begin{proof}
Let $a$ be an arbitrary hexagon of $\mathcal{H}$. Let the family $\{\mathcal{P}_i\}_{i=0}^{\infty}$
be as defined in the proof of Lemma~\ref{lemma:adjacencylemma}. Since $\mathcal{K}$ is finite,
there exists $n \in \mathbb{N}$ such that $\mathcal{K} \subseteq \mathcal{P}_n$. (More precisely,
for every $k \in \mathcal{K}$ there exists $n_k$ such that $k \in \mathcal{P}_{n_k}$. Take
$n \coloneqq \max_{k \in \mathcal{K}} n_k$.) The complement of $\mathcal{P} \coloneqq \mathcal{P}_{n}$
is contained in $\mathcal{K}^\complement$. Because $\mathcal{P}^\complement$ is connected, it
is contained in a connected components of $\mathcal{K}^\complement$. Denote this component by
$\mathcal{C}_\infty$. Note that this is the infinite component. In addition to $\mathcal{C}_\infty$,
$\mathcal{K}^\complement$ may have more connected components. All of them (if there are any)
are contained in $\mathcal{P}$. Each is finite, because $\mathcal{P}$ is finite. Their number
is bounded by the number of hexagons in $\mathcal{P}$. Therefore, $\mathcal{K}^\complement$ has
finitely many connected components.

The fact that all finite components are coronoids is clear from the definition of coronoids.
\end{proof}

\noindent The above lemma does not apply to infinite hexagonal systems.
(See Figure \ref{fig:infinitelem5} for examples.)

\begin{figure}[!hbt]
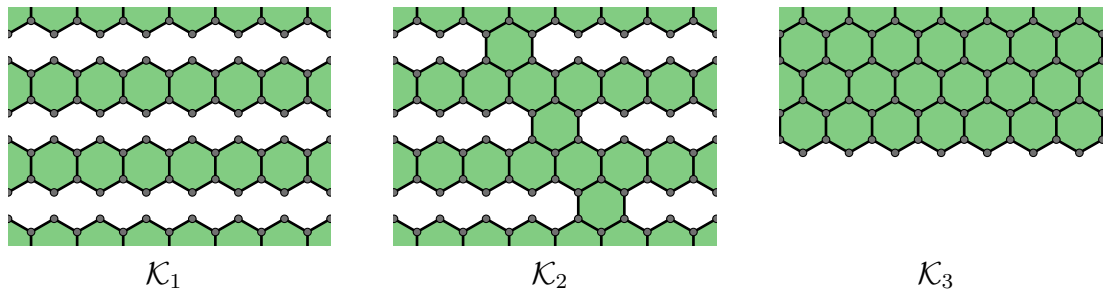

\centering
\begin{tabular}{ccc}
\input{img/coronoid_inf_01.tikz} & \input{img/coronoid_inf_02.tikz} & \input{img/coronoid_inf_03.tikz} \\
$\mathcal{K}_1$ & $\mathcal{K}_2$  & $\mathcal{K}_3$ 
\end{tabular}
\caption{$\mathcal{K}_1$ consists of infinitely many disjoint infinite lines of hexagons. 
Its complement has infinitely many connected components that are themselves infinite.
$\mathcal{K}_2$ is obtained from $\mathcal{K}_1$ by adding another line of hexagons with
a different slope. The hexagonal system $\mathcal{K}_2$ is a connected example. Hexagonal
system $\mathcal{K}_3$ (half-plane) is infinite and its complement $\mathcal{K}_3^\complement$
has a single connected component. In fact, a connected hexagonal system with arbitrary many
finite and arbitrary many infinite connected components can be obtained.}
\label{fig:infinitelem5}
\end{figure}

\begin{lemma}
\label{lemma:pathlemma}
Let $\mathcal{K}$ be a coronoid and let $\mathcal{C}_a$ and $\mathcal{C}_b$ be two connected
components of its complement $\mathcal{K}^\complement$. Let $a \in \mathcal{C}_a$ and
$b \in \mathcal{C}_b$. Then $a \equiv_{\mathcal{K} \cup \mathcal{C}_a \cup \mathcal{C}_b} b$.
\end{lemma}
\noindent Components $\mathcal{C}_a$ and $\mathcal{C}_b$ in the above lemma may be either two corona holes
or one corona hole and the exterior of $\mathcal{K}$. 

\begin{proof}
By Lemma~\ref{lemma:adjacencylemma} there exists $\tilde{a} \in \mathcal{C}_a$, such that
$\tilde{a} \in N(k_a)$ for some $k_a \in \mathcal{K}$. By the same lemma, there exists
$\tilde{b} \in \mathcal{C}_b$, such that $\tilde{b} \in N(k_b)$ for some $k_b \in \mathcal{K}$.
Since $\mathcal{C}_a$ is a connected component, $a \equiv_{\mathcal{C}_a} \tilde{a}$.
Similarly, $b \equiv_{\mathcal{C}_b} \tilde{b}$. From $\tilde{a} \equiv_{\mathcal{K} \cup \mathcal{C}_a} k_a$,
$\tilde{b} \equiv_{\mathcal{K} \cup \mathcal{C}_b} k_b$ and $k_a \equiv_{\mathcal{K}} k_b$,
it follows that $a \equiv_{\mathcal{K} \cup \mathcal{C}_a \cup \mathcal{C}_b} b$.
\end{proof}

\begin{corollary}
\label{cor:pathcor}
Let $\mathcal{K}$ be a coronoid and let $\mathcal{C}_a$ be a connected component of its
complement $\mathcal{K}^\complement$. Let $a \in \mathcal{C}_a$ and $k \in \mathcal{K}$.
Then $a \equiv_{\mathcal{K} \cup \mathcal{C}_a} k$.
\end{corollary}

\begin{proof}
In the proof of Lemma~\ref{lemma:pathlemma} we have already shown the existance of
$\tilde{a} \in \mathcal{C}_a$, such that $\tilde{a} \in N(k_a)$ for some $k_a \in \mathcal{K}$.
From $a \equiv_{\mathcal{C}_a} \tilde{a}$, $\tilde{a} \equiv_{\mathcal{K} \cup \mathcal{C}_a} k_a$
and $k_a \equiv_{\mathcal{K}} k$, it follows that $a \equiv_{\mathcal{K} \cup \mathcal{C}_a} k$.
\end{proof}

\begin{theorem}
\label{thm:decompositionthm}
Let $\mathcal{K}$ be a coronoid. The complement $\mathcal{K}^\complement$ of $\mathcal{K}$ has a
finite number $b(\mathcal{K}) \coloneqq d + 1$, $d \geq 0$, of connected components
$\mathcal{B}_\infty, \mathcal{B}_1, \mathcal{B}_2, \ldots, \mathcal{B}_d$ such that
\begin{equation}
\mathcal{K}^\complement = \mathcal{B}_\infty \sqcup \mathcal{B}_1 \sqcup \mathcal{B}_2 \sqcup \cdots \sqcup \mathcal{B}_d.
\label{eq:decomp}\tag{$\star$}
\end{equation}
Exactly one component, denoted $\mathcal{B}_\infty$, is infinite and the other $d$ components
are finite, each being a benzenoid. Moreover, $\mathcal{B}_\infty^\complement$ is also a benzenoid.
\end{theorem}

\begin{proof}
From Lemma~\ref{lemma:decomposition}, it immediately follows that
$\mathcal{K}^\complement = \mathcal{B}_\infty \sqcup \mathcal{B}_1 \sqcup \mathcal{B}_2 \sqcup \cdots \sqcup \mathcal{B}_d$,
where $\mathcal{B}_\infty$ is an infinite and $\mathcal{B}_i$, $1 \leq i \leq d$, are
finite connected components. We need to show that each $\mathcal{B}_i$, $1 \leq i \leq d$,
is a benzenoid.

It only remains to show that the complement $\mathcal{B}_i^\complement$ of $\mathcal{B}_i$,
$1 \leq i \leq d$, is connected. From
$\mathcal{H} = \mathcal{K} \sqcup \mathcal{K}^\complement = \mathcal{K} \sqcup
\left( \mathcal{B}_\infty \sqcup \mathcal{B}_1 \sqcup \mathcal{B}_2 \sqcup \cdots \sqcup \mathcal{B}_d \right)$
it follow that
$$
\mathcal{B}_i^\complement = \mathcal{K} \sqcup \mathcal{B}_\infty \sqcup \mathcal{B}_1 \sqcup \mathcal{B}_2 \sqcup
\cdots \sqcup \mathcal{B}_{i-1} \sqcup \mathcal{B}_{i+1} \sqcup \cdots \sqcup \mathcal{B}_{d}.
$$
Hexagonal systems $\mathcal{K}, \mathcal{B}_\infty, \mathcal{B}_1, \ldots, \mathcal{B}_{i-1}, \mathcal{B}_{i+1}, \ldots, \mathcal{B}_{d}$
are all connected. By Lemma~\ref{lemma:pathlemma} and Corollary~\ref{cor:pathcor}, their
union is also connected.

To show that $\mathcal{B}_\infty^\complement$ is a benzenoid, we need to show that
$\mathcal{B}_\infty^\complement$ and $(\mathcal{B}_\infty^\complement)^\complement$ are
connected and that $\mathcal{B}_\infty^\complement$ is finite. Since
$(\mathcal{B}_\infty^\complement)^\complement = \mathcal{B}_\infty$, it is clearly connected.
Since
$$
\mathcal{B}_\infty^\complement = \mathcal{K} \sqcup \mathcal{B}_1 \sqcup \mathcal{B}_2 \sqcup \cdots \sqcup \mathcal{B}_{d},
$$
it is connected by the same argument that worked for $\mathcal{B}_i^\complement$ above.
Hexagonal systems $\mathcal{K}, \mathcal{B}_1, \ldots, \mathcal{B}_{d}$ are all finite and
therefore their union $\mathcal{B}_\infty^\complement$ is also finite.
\end{proof}

\begin{definition}
Let $\mathcal{K}$ be a coronoid. The \emph{benzenoid closure} of $\mathcal{K}$, denoted
$\overline{\mathcal{K}}$, is the intersection of all those benzenoids that include $\mathcal{K}$
as a subset, i.e.,
\begin{equation}
\label{eq:benzoclosure}\tag{$\dagger$}
\overline{\mathcal{K}} = \bigcap\, \{ \mathcal{B} \mid \mathcal{B}\text{ is benzenoid} \land \mathcal{K} \subseteq \mathcal{B} \}.
\end{equation}
\end{definition}

\noindent For a coronoid $\mathcal{K}$ we define 
$\mathit{Benz}(\mathcal{K}) = \{ \mathcal{B} \mid \mathcal{B}\text{ is benzenoid} \land \mathcal{K} \subseteq \mathcal{B} \}$.
This set will be repeatedly used in several proofs that follow. Using this terminology,
(\ref{eq:benzoclosure}) from the above definition can be written in a shorter form as
$\overline{\mathcal{K}} = \bigcap \mathit{Benz}(\mathcal{K})$.

\begin{lemma}
\label{lemma:closure}
Benzenoid closure $\overline{\mathcal{K}}$ of a (general) coronoid $\mathcal{K}$ is
a benzenoid. Moreover,
$\overline{\mathcal{K}}=\mathcal{B}_\infty^\complement = \mathcal{K} \sqcup \mathcal{B}_1 \sqcup \cdots \sqcup \mathcal{B}_d$,
where
$\mathcal{K}^\complement = \mathcal{B}_\infty \sqcup \mathcal{B}_1 \sqcup \mathcal{B}_2 \sqcup \cdots \sqcup \mathcal{B}_d$
as in Theorem~\ref{thm:decompositionthm}.
\end{lemma}

\begin{proof}
By Theorem~\ref{thm:decompositionthm},
$\mathcal{K}^\complement = \mathcal{B}_\infty \sqcup \mathcal{B}_1 \sqcup \mathcal{B}_2 \sqcup \cdots \sqcup \mathcal{B}_d$,
where $\mathcal{B}_\infty$ is infinite and $\mathcal{B}_i$, $1 \leq i \leq d$, are
benzenoids. We will show that
$\overline{\mathcal{K}} = \mathcal{K} \cup \mathcal{B}_1 \cup \cdots \cup \mathcal{B}_{d} = \mathcal{B}_\infty^\complement$.

Let $\mathcal{L}$ be an arbitrary benzenoid such that $\mathcal{K} \subseteq \mathcal{L}$.
Then $\mathcal{L}^\complement \subseteq \mathcal{K}^\complement$. Because $\mathcal{L}^\complement$
is connected, it is contained in at most one of components
$\mathcal{B}_\infty, \mathcal{B}_1, \mathcal{B}_2, \ldots, \mathcal{B}_d$. Since
$\mathcal{L}^\complement$ is infinite, $\mathcal{L}^\complement \subseteq \mathcal{B}_\infty$.
Therefore,
$$
\mathcal{K} \cup \mathcal{B}_1 \cup \mathcal{B}_2 \cup \cdots \cup \mathcal{B}_d =
\mathcal{B}_\infty^\complement \subseteq (\mathcal{L}^\complement)^\complement = \mathcal{L}.
$$
This implies $\mathcal{K} \cup \mathcal{B}_1 \cup \mathcal{B}_2 \cup \cdots \cup \mathcal{B}_d \subseteq \overline{\mathcal{K}}$.

By Theorem~\ref{thm:decompositionthm},
$\mathcal{B}_\infty^\complement = \mathcal{K} \cup \mathcal{B}_1 \cup\cdots \cup \mathcal{B}_d$
is a benzenoid. Thus $\overline{\mathcal{K}} \subseteq \mathcal{K} \cup \mathcal{B}_1 \cup \cdots \cup \mathcal{B}_d$. 
We have proved that $\overline{\mathcal{K}} = \mathcal{K} \cup \mathcal{B}_1 \cup \cdots \cup \mathcal{B}_d$,
where $\mathcal{K} \cup \mathcal{B}_1 \cup \cdots \cup \mathcal{B}_d$ is a benzenoid.
This proves existence, and also uniqueness, of $\overline{\mathcal{K}}$.
\end{proof}

\noindent In the above proof of Lemma~\ref{lemma:closure} we have also shown how to construct
$\overline{\mathcal{K}}$ for a given $\mathcal{K}$. Reader should note that in general
case the intersection of two benzenoids is not necessarily a benzenoid
(see Figure~\ref{fig:notbenzenoid}).

\begin{proposition}
\label{prop:propertiesofcl}
The benzenoid closure $\mathcal{K} \mapsto \overline{\mathcal{K}}$ is an operation on
the set of all coronoids that satisfies the following three conditions:
\begin{enumerate}
\renewcommand{\labelenumi}{(\alph{enumi})}
\item $\forall \mathcal{K} \colon \mathcal{K} \subseteq \overline{\mathcal{K}}$,
\item $\forall \mathcal{K}, \mathcal{L} \colon \mathcal{K} \subseteq \mathcal{L} \implies \overline{\mathcal{K}} \subseteq \overline{\mathcal{L}}$, and
\item $\forall \mathcal{K} \colon \overline{\overline{\mathcal{K}}} = \overline{\mathcal{K}}$.
\end{enumerate} 
\end{proposition}

\noindent Note that the co-domain of the mapping $\mathcal{K} \mapsto \overline{\mathcal{K}}$ can be
restricted to the set of all benzenoids. This mapping is surjective and the preimage of every
benzenoid is a finite set of coronoids.

\begin{proof}
By definition, $\overline{\mathcal{K}} = \bigcap \mathit{Benz}(\mathcal{K})$.
From the definition it follows directly that $\mathcal{K} \subseteq \overline{\mathcal{K}}$.

Let us now show that $\overline{\overline{\mathcal{K}}} = \overline{\mathcal{K}}$.
By definition, $\overline{\overline{\mathcal{K}}} = \bigcap \mathit{Benz}(\overline{\mathcal{K}})$.
Therefore $\overline{\mathcal{K}} \subseteq \overline{\overline{\mathcal{K}}}$. Since
$\overline{\mathcal{K}}$ is a benzenoid, $\overline{\mathcal{K}} \in \mathit{Benz}(\overline{\mathcal{K}})$.
Therefore, $\overline{\overline{\mathcal{K}}} \subseteq \overline{\mathcal{K}}$.

Finally, we will show that
$\mathcal{K} \subseteq \mathcal{L} \implies \overline{\mathcal{K}} \subseteq \overline{\mathcal{L}}$.
By definition, $\overline{\mathcal{K}} = \bigcap \mathit{Benz}(\mathcal{K})$ and
$\overline{\mathcal{L}} = \bigcap \mathit{Benz}(\mathcal{L})$.
From $\mathcal{K} \subseteq \mathcal{L}$ it follows that every element of
$\mathit{Benz}(\mathcal{L})$ is also an element of $\mathit{Benz}(\mathcal{K})$, i.e.,
$\mathit{Benz}(\mathcal{L}) \subseteq \mathit{Benz}(\mathcal{K})$. Therefore
$\overline{\mathcal{K}} \subseteq \overline{\mathcal{L}}$.
\end{proof}

\noindent We define another closure operation.

\begin{definition}
\label{def:alternativeclosure}
An \emph{alternative benzenoid closure} $\mathcal{K} \mapsto \widetilde{\mathcal{K}}$
is a mapping on the set of all coronoids that satisfies the following three conditions:
\begin{enumerate}
\renewcommand{\labelenumi}{(\roman{enumi})}
\item $\forall \mathcal{K} \colon \widetilde{\mathcal{K}}$ is benzenoid,
\item $\forall \mathcal{K} \colon \mathcal{K} \subseteq \widetilde{\mathcal{K}}$, and
\item $\forall \mathcal{K} \colon \forall \text{benzenoid }\mathcal{L} \colon \mathcal{K} \subseteq
\mathcal{L} \implies \widetilde{\mathcal{K}} \subseteq \mathcal{L}$.
\end{enumerate}
\end{definition}

\noindent The following lemma tells us that this corresponds to an alternative definition of
the benzenoid closure operation.

\begin{lemma}
\label{lem:alternativeclop}
Let $\mathcal{K}$ be a general coronoid. Then $\widetilde{\mathcal{K}} = \overline{\mathcal{K}}$.
With other words, benzenoid closure operation and the alternative benzenoid closure
operation coincide.
\end{lemma}

\begin{proof}
First, we will prove existence of $\widetilde{\mathcal{K}}$ by proving that
$\overline{\mathcal{K}}$ satisfies the three conditions of
Definition~\ref{def:alternativeclosure}.

By Lemma \ref{lemma:closure}, $\overline{\mathcal{K}}$ is a benzenoid, so (i) holds.
Proposition~\ref{prop:propertiesofcl} tells us that $\mathcal{K} \subseteq \overline{\mathcal{K}}$,
so (ii) also holds. It only remains to see that
$\mathcal{K} \subseteq \mathcal{L} \implies \overline{\mathcal{K}} \subseteq \mathcal{L}$
for every benzenoid $\mathcal{L}$.

Let $\mathcal{L}$ be any benzenoid such that $\mathcal{K} \subseteq \mathcal{L}$. By
definition, $\mathcal{L} \in \mathit{Benz}(\mathcal{K})$. It is clear that
$\mathcal{L} \in \mathit{Benz}(\mathcal{K})$ implies
$\overline{\mathcal{K}} = \bigcap \mathit{Benz}(\mathcal{K}) \subseteq \mathcal{L}$.

In principle, there could exist some other benzenoid, different from $\overline{\mathcal{K}}$,
which would also satisfy the three conditions in Definition~\ref{def:alternativeclosure}.
We will show that this is not the case by proving that every $\widetilde{\mathcal{K}}$ from
Definition~\ref{def:alternativeclosure} is $\overline{\mathcal{K}}$. 

Let $\mathcal{L}$ be any element of $\mathit{Benz}(\mathcal{K})$. Condition (iii) implies
that $\widetilde{\mathcal{K}} \subseteq \mathcal{L}$. Therefore
$\widetilde{\mathcal{K}} \subseteq \bigcap \mathit{Benz}(\mathcal{K}) = \overline{\mathcal{K}}$.
By condition (i) and (ii), $\widetilde{\mathcal{K}}$ is a benzenoid such that
$\mathcal{K} \subseteq \widetilde{\mathcal{K}}$. This means that $\widetilde{\mathcal{K}}$
is a member of $\mathit{Benz}(\mathcal{K})$. Clearly,
$\overline{\mathcal{K}} = \bigcap \mathit{Benz}(\mathcal{K}) \subseteq \widetilde{\mathcal{K}}$.
From $\widetilde{\mathcal{K}} \subseteq \overline{\mathcal{K}}$ and
$\overline{\mathcal{K}} \subseteq \widetilde{\mathcal{K}}$ it follows that
$\widetilde{\mathcal{K}} = \overline{\mathcal{K}}$ which completes the proof.
\end{proof}

\begin{figure}[!h]
\centering
\input{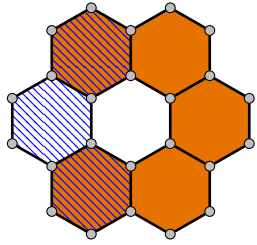}
\caption{The intersection of two benzenoids is not necessarily a benzenoid.}
\label{fig:notbenzenoid}
\end{figure}

\begin{lemma}
\label{lemma:intersectionofbenzenoids}
The intersection of two benzenoids is the disjoint union of finitely many benzenoids.
\end{lemma}

\begin{proof}
Let $\mathcal{L}_a$ and $\mathcal{L}_b$ be two benzenoids and let
$\mathcal{L} = \mathcal{L}_a \cap \mathcal{L}_b$. Because $\mathcal{L}_a$ and $\mathcal{L}_b$
are finite, $\mathcal{L}$ is also finite. It consists of finitely many (possibly 0) finite
connected components. That allows us to write
$\mathcal{L} = \mathcal{K}_1 \sqcup \mathcal{K}_2 \sqcup \cdots \sqcup \mathcal{K}_d$.
To complete the proof, we have to show that each $\mathcal{K}_i^\complement$ is connected. 

By Lemma~\ref{lemma:decomposition},
$\mathcal{K}_i^\complement = \mathcal{C}_\infty \sqcup \mathcal{C}_1 \sqcup \mathcal{C}_2 \sqcup \cdots \sqcup \mathcal{C}_m$.
From $\mathcal{K}_i \subseteq \mathcal{L} \subseteq \mathcal{L}_a$ we obtain
$\mathcal{L}_a^\complement \subseteq \mathcal{K}_i^\complement$. Since
$\mathcal{L}_a^\complement$ is connected and infinite,
$\mathcal{L}_a^\complement \subseteq \mathcal{C}_\infty$.
In addition to $\mathcal{C}_\infty$, $\mathcal{K}_i^\complement$ may have 0 or more
other connected components. Suppose $m \geq 1$.

By Lemma~\ref{lemma:adjacencylemma}, there exists $h \in \mathcal{C}_1$ which is adjacent
to some $k \in \mathcal{K}_i$. Then $h \notin \mathcal{L}_a$ or $h \notin \mathcal{L}_b$.
If $h$ belonged to both $ \mathcal{L}_a$ and $\mathcal{L}_b$, then $h \in \mathcal{K}_j$ for
some $j$. From $h \in \mathcal{C}_1 \subseteq \mathcal{K}_i^\complement$, it follows
that $j \neq i$. But this contradicts Lemma~\ref{lemma:neighbours}, so $h$ indeed belongs
to at most one of $ \mathcal{L}_a$ or $\mathcal{L}_b$. Without loss of generality, we can
assume $h \notin \mathcal{L}_a$. In other words,
$h \in \mathcal{L}_a^\complement \subseteq \mathcal{C}_\infty$. This contradicts the fact
that $h \in \mathcal{C}_1$.

This means that $m = 0$, implying $\mathcal{K}_i^\complement = \mathcal{C}_\infty$, and the
proof is complete.
\end{proof}

\begin{figure}[!htb]
\centering
\input{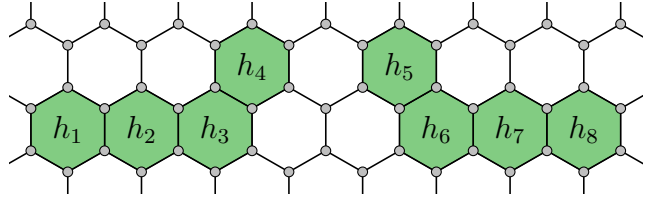}
\caption{Benzenoids $\mathcal{K}_1 = \{h_1, h_2, h_3, h_4 \}$ and
$\mathcal{K}_2 = \{h_5, h_6, h_7, h_8 \}$ are equivalent.}
\label{fig:eqivalentCoronoidz}
\end{figure}

\noindent Look at Figure~\ref{fig:eqivalentCoronoidz}. Benzenoids $\mathcal{K}_1$ and $\mathcal{K}_2$ are
not equal, i.e., $\mathcal{K}_1 \neq \mathcal{K}_2$. For instance, $h_1 \in \mathcal{K}_1$ but
$h_1 \notin \mathcal{K}_2$. If one would draw them on a piece of paper and cut them out, they
would coincide. This notion of coincidence can be precisely defined. Let $\Aut(\mathcal{H})$ be
the group of symmetries of the hexagonal grid, i.e., isometries of the plane that map hexagons
to hexagons. (For more details on this topic see \cite{grunbaum2013tilings}.) If the coordinate
system is placed so that the origin is in the centre of a chosen hexagon and if hexagons have
sides of unit length then isometries $\phi_1, \phi_2, \phi_3, \phi_4 \colon \mathbb{R}^2 \to \mathbb{R}^2$,
where
\begin{align*}
\phi_1(x, y) & = (x + \sqrt{3}, y),\\
\phi_2(x, y) & = (x + {\textstyle \frac{\sqrt{3}}{2}}, y + {\textstyle \frac{3}{2}}), \\
\phi_3(x, y) &= (-x, y), \text{ and }\\
\phi_4(x, y) &= ({\textstyle x \cdot \cos \frac{\pi}{3} - y \cdot \sin \frac{\pi}{3}, x \cdot \sin \frac{\pi}{3} + y \cdot \cos \frac{\pi}{3}})
\end{align*}
generate the group $\Aut(\mathcal{H})$. Now, we can define:

\begin{definition}
Hexagonal systems $\mathcal{K}$ and $\mathcal{L}$ are \emph{equivalent}, denoted
$\mathcal{K} \cong \mathcal{L}$, if there exists an isometry $\psi \in \Aut(\mathcal{H})$, 
such that $\psi(\mathcal{K}) = \mathcal{L}$.
\end{definition}

\noindent Let $\mathcal{K}$ be a (general) coronoid. From Lemma~\ref{lemma:closure} it follows that
$$
\overline{\mathcal{K}} \setminus \mathcal{K} = \mathcal{B}_1 \sqcup \cdots \sqcup \mathcal{B}_d \quad (d \geq 0),
$$
where each $\mathcal{B}_i$ is a benzenoid. The benzenoids $\mathcal{B}_i$ are called
\emph{corona holes}. They consist of one or more hexagons. The (general) coronoids (a) and
(b) in Figure \ref{fig:coroexamples} have two corona holes each, whilst coronoid (c) has
only one. Examples (b) and (c) are special in a sense. A general coronoid is called
\emph{degenerate} if one of its corona holes is a single hexagon. This is because such a
corona hole has no interpretation in chemistry. Otherwise it is called \emph{non-degenerate}.
Corona holes of size 1 will be called \emph{degenerate corona holes}.

\begin{definition}
Let $\mathcal{K}$ be any coronoid. Non-degenerate closure of $\mathcal{K}$, denoted
$\mathit{NonDeg}(\mathcal{K})$, is the smallest non-degenerate coronoid which includes
$\mathcal{K}$.
\end{definition}

\noindent It is not hard to see that $\mathit{NonDeg}(\mathcal{K})$ can be obtained from $\mathcal{K}$
by adding to it exactly its degenerate corona holes. Later in this paper, where we consider
applications of this theory in chemistry, by the word \emph{coronoid} we will always mean
a non-degenerate coronoid.

The definition of coronoids includes benzenoids as a special case. A \emph{proper coronoid}
is one that is not also a benzenoid. Using the terminology of topology one may say that the
only difference between benzenoids and proper coronoids is that both are connected, but the
former are also simply connected. To give a precise meaning to this, one has to consider
benzenoids as surfaces with a boundary. Plane graphs on their own are not sufficient model for
coronoids. Some faces have to be labelled as ``not present''. We will call them \emph{holes}.
In this context, the outer face of a coronoid is merely one of its holes. Chemists do not always
distinguish between the two models (coronoids as plane graphs and coronoids as surfaces with
boundary); for them, exactly those faces of length strictly greater than 6 are holes. In other
words, they do not recognise degenerate corona holes. As we will see later, when we generalise
coronoids to perforated patches, some care should be taken over this distinction.

We defined coronoids (and benzenoids) as subsets of hexagons in the Euclidean plane. We may
consider them as 2-dimensional \emph{cell complexes} \cite{hatcher2002algebraic}. In what follows,
we will adopt some topological terminology and notation. The infinite hexagonal grid $\mathcal{H}$
can be obtained by embedding the infinite cubic graph called \emph{hexagonal lattice}
\cite{hammack2011handbook} in the Euclidean plane. Vertices of the hexagonal lattice are 0-cells,
edges are 1-cells and faces are 2-cells. Every edge of the hexagonal lattice is incident to exactly
two distinct hexagonal faces. Any two distinct faces can either share a single edge or nothing at all.

Let $\mathcal{K}$ be a coronoid. By our definition it is a collection of 2-cells. We can assign
to $\mathcal{K}$ the smallest \emph{subcomplex} of the hexagonal grid which includes all hexagons
(2-cells) of $\mathcal{K}$. Note that a subcomplex contains the closure of each of its cells. Its
1-skeleton is a graph that is a subgraph of the hexagonal lattice consisting of those vertices and
edges which are incident with at least one hexagon of $\mathcal{K}$. We will denote this graph by
$G(\mathcal{K})$ and call it a \emph{skeleton} in the general setting and a \emph{coronoid graph}
when we deal with coronoids. Note that if $\mathcal{H}$ denotes the hexagonal grid, then
$G(\mathcal{H})$ is the hexagonal lattice. Note that $\Aut(G(\mathcal{H}))$ acts transitively on
vertices of $G(\mathcal{H})$. Elements of $\Aut(G(\mathcal{H}))$ are graph automorphisms.

\begin{figure}[!hb]
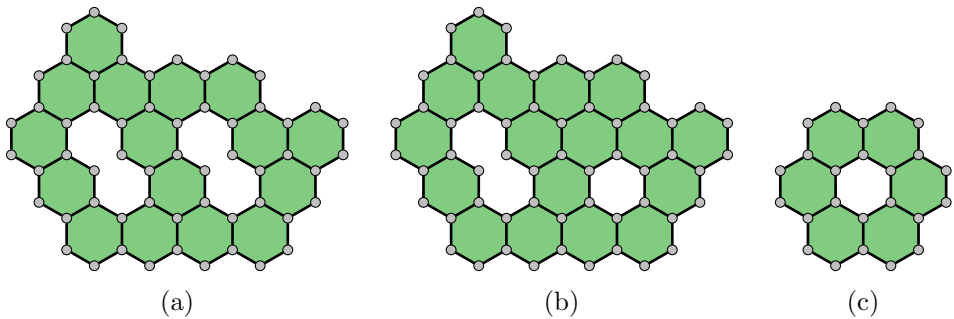

\centering
\subfigure[]{
\input{img/coronoid01.tikz}
}
\subfigure[]{
\input{img/coronoid01a.tikz}
}
\subfigure[]{
\input{img/coronoid02.tikz}
}
\caption{Three examples of general coronoids. The first is non-degenerate whilst the second
and third are degenerate.}
\label{fig:coroexamples}
\end{figure}
Consider the three examples of coronoids on Figure \ref{fig:coroexamples}.
As in the case of benzenoids, every edge is incident to exactly two distinct faces of which
one may be a hole, but not both. Every vertex is incident to either 2 or 3 edges. A coronoid
graph is a 2-connected and 2-edge-connected graph.

The edges of a coronoid graph are naturally divided into two types: internal and boundary. An edge
belonging to two (adjacent) hexagons is \emph{internal} and an edge belonging to only one hexagon
is a \emph{boundary edge}. Vertices of a coronoid graph can also be divided into internal and boundary.
\emph{Internal} vertices are incident with 3 internal edges. All other vertices of $G(\mathcal{K})$ are
called \emph{boundary}. Boundary vertices can be further divided into two types: those of degree 2 and
those of degree 3. A hexagon $h$ of a coronoid $\mathcal{K}$ is \emph{internal} if $|N_\mathcal{K}(h)| = 6$.
Otherwise it is a \emph{boundary} hexagon. In other words, an internal hexagon is surrounded by 6
internal edges. The subgraph of $G(\mathcal{K})$ that consists of all boundary vertices and edges will
be called \emph{border} of $\mathcal{K}$ and denoted $\partial \mathcal{K}$. When $\mathcal{K} = \{k\}$
is a singleton, we will write $\partial k$. A benzenoid is called \emph{catacondensed} if it has no
internal vertices, and \emph{pericondensed} otherwise. Pericondensed benzenoids can be further
divided into two groups: those with internal hexagons which we call \emph{corpulent} benzenoids and
those without internal hexagons which we call \emph{gaunt} benzenoids.

\begin{proposition}
\label{prop:girthoflattice6}
The girth of the infinite hexagonal lattice is 6, i.e., $$\girth(G(\mathcal{H})) = 6.$$
Let $C$ be an arbitrary cycle of $G(\mathcal{H})$. Then $|C| = 6$ if and only if there
exists a hexagon $h \in \mathcal{H}$ such that $C = \partial h$.
\end{proposition}

\begin{proof}
Let $h \in \mathcal{H}$ be any hexagon. Then $\partial h \cong C_6$. We will show that there are no
shorter cycles and also that all 6-cycles are the borders of hexagonal faces of the hexagonal grid.

The distance between vertices $u$ and $v$ of a graph $G$ is the length of the shortest path
between $u$ and $v$ in $G$. We will denote it by $d_G(u, v)$.

Owing to symmetry, to investigate the structural properties of cycles of $G(\mathcal{H})$ it is
enough to investigate cycles that contain a given vertex $u$. Up to symmeties of $G(\mathcal{H})$
there is only one path of length 2. See Figure~\ref{fig:paths2}.
\begin{figure}[!ht]
\centering
\input{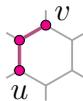}
\caption{The only type of path of length 2 in $G(\mathcal{H})$.}
\label{fig:paths2}
\end{figure}
Vertices $u$ and $v$ are not adjacent, so there are no triangles in $G(\mathcal{H})$. There are two
possible paths of length 3 (see Figure~\ref{fig:paths2a}).
\begin{figure}[!ht]
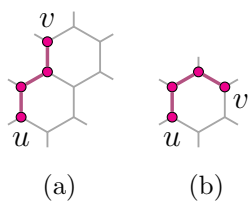

\centering
\subfigure[]{
\input{img/path3a.tikz}
}
\subfigure[]{
\input{img/path3b.tikz}
}
\caption{The two types of paths of length 3 in $G(\mathcal{H})$.}
\label{fig:paths2a}
\end{figure}
In both the endpoints are not adjacent, so there are no rectangles in $G(\mathcal{H})$.
Both paths of length 3 can be extended in two ways yielding four paths of length 4
(see Figure~\ref{fig:paths2b}).
\begin{figure}[!ht]
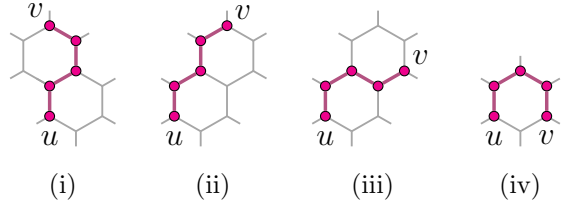

\centering
\renewcommand{\thesubfigure}{(\roman{subfigure})}
\subfigure[]{
\input{img/path4a.tikz}
}
\subfigure[]{
\input{img/path4b.tikz}
}
\subfigure[]{
\input{img/path4c.tikz}
}
\subfigure[]{
\input{img/path4d.tikz}
}
\caption{The four types of paths of length 4 in $G(\mathcal{H})$.}
\label{fig:paths2b}
\end{figure}
In all cases their endpoints are non-adjacent, so there are no 5-cycles in $G(\mathcal{H})$.
In cases (i), (ii) and (iii) we have that $d(u, v) = 4$. Therefore, these paths cannot
form 6-cycles. The case (iv) can be extended to a 5-path in two ways (see Figure~\ref{fig:paths2c}).
\begin{figure}[!ht]
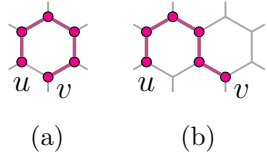

\centering
\subfigure[]{
\input{img/path5a.tikz}
}
\subfigure[]{
\input{img/path5b.tikz}
}
\caption{Paths of length 5 in $G(\mathcal{H})$.}
\label{fig:paths2c}
\end{figure}
In case (b), we have that $d(u, v) = 3$, which means that no 6-cycle can be formed.
In case (a), vertices $u$ and $v$ are adjacent and therefore form a 6-cycle. We have
examined all options and no other 6-cycles arise, which means that all of them can be
obtained as borders of hexagons. 
\end{proof}

\noindent Note that every vertex of $G(\mathcal{K})$ belongs to one, two or three 6-cycles and
that every edge of $G(\mathcal{K})$ belongs to one or two 6-cycles. Because
$G(\mathcal{K}) \subseteq G(\mathcal{H})$, every 6-cycle of $G(\mathcal{K})$ is also
a 6-cycle of $G(\mathcal{H})$. Therefore, for every 6-cycle $C$ in $G(\mathcal{K})$
there exists some $h \in \mathcal{H}$, such that $C = \partial h$. Let $u$ be an arbitrary
vertex of $G(\mathcal{K})$. It is incident with 3 hexagons, say $h_1, h_2$ and $h_3$,
of $\mathcal{H}$. By definition of skeleton, at least one of those hexagons must also
be in $\mathcal{K}$. Without loss of generality $h_1 \in \mathcal{K}$. Therefore,
$u$ belongs to 6-cycle $\partial h_1$. It may also happen that $h_2 \in \mathcal{K}$
and/or $h_3 \in \mathcal{K}$. Then vertex $u$ also belongs to cycle $\partial h_2$
and/or $\partial h_3$. Let $e$ be an arbitrary edge of $G(\mathcal{K})$. There exist
$h_1 \in \mathcal{H}$ and $h_2 \in \mathcal{H}$ such that $e \in \partial h_1$ and
$e \in \partial h_2$. By definition of skeleton at least one of them also belongs to
$\mathcal{K}$. It may also happen that both of them belong to $\mathcal{K}$. 
The following proposition is also obvious:

\begin{proposition}
Let $\mathcal{K} \subseteq \mathcal{H}$ be an arbitrary coronoid. Then the following
statements are true:
 \begin{enumerate}
\renewcommand{\labelenumi}{(\alph{enumi})}
\item Graph $G(\mathcal{K})$ is connected.
\item For every edge $e \in E(G(\mathcal{K}))$ there exists $h \in \mathcal{K}$ such
that $e \in \partial h$.
\item For every cycle $C \subseteq G(\mathcal{K})$ the inequality $|C| \geq 6$ holds.
The equality is attained if and only if there exists a hexagon $h \in \mathcal{H}$ such
that $C = \partial h$. Moreover, if $\mathcal{K}$ is non-degenerate then $h \in \mathcal{K}$.
\end{enumerate}
\end{proposition}

\begin{proof}
To show that $G(\mathcal{K})$ is connected, we will find a $(u,v)$-path for any pair
of vertices $u, v \in V(G(\mathcal{K}))$. We already know that there exist $h_u$ and
$h_v$ such that $u \in \partial h_u$ and $v \in \partial h_v$. There exists a sequence
of hexagon $h_0 = h_v, h_1, \ldots, h_n = h_u$ such that $h_i$ and $h_{i+1}$ are adjacent
for all $i$. If two hexagons $h$ and $k$ are adjacent, there exists a path from any
vertex of $\partial h$ to any vertex of $\partial k$. On every hexagon $h_i$, $1 \leq i < n$,
choose a vertex $v_i$. Let $v_0 = v$ and $v_n = u$. There exist paths $P_i$ with
endvertices $v_i$ and $v_{i+1}$. Concatenation of paths $P_0, P_1, \ldots, P_{n-1}$ is
a $(u, v)$-walk.

Statement (b) was already in the discussion above.

If $H \subseteq G$, then $\girth(H) \geq \girth(G)$. The fact that
$G(\mathcal{K}) \subseteq G(\mathcal{H})$ and Proposition~\ref{prop:girthoflattice6} give
us the inequality. Those cycles which attain the equality are characterized by
Proposition~\ref{prop:girthoflattice6}. Now suppose that $h \notin \mathcal{K}$ and
that $\mathcal{K}$ is non-degenerate. Hexagon $h$ is sorrounded by 6 hexagons
$h_1, h_2, \ldots, h_6$ of $\mathcal{H}$. By definition of $G(\mathcal{K})$, at least
one of $h$ or $h_i$ must be in $\mathcal{K}$ for all i. This means that
$h_1, \ldots, h_6 \in \mathcal{K}$. But then $\{h\}$ is a corona hole of size 1,
a contradiction.
\end{proof}

\begin{proposition}
\label{prop:prop1}
For any coronoid $\mathcal{K}$, the subgraph $\partial \mathcal{K}$ forms a collection
of $b(\mathcal{K}) \geq 1$ cycles. A coronoid is a benzenoid if $b(\mathcal{K}) = 1$
and is a proper coronoid with $b(\mathcal{K}) - 1$ corona holes if $b(\mathcal{K}) > 1$.
\end{proposition}

\begin{proof}
Look at Figure~\ref{fig:typesofboundaryvert}. There are two types of boundary vertex:
the one in Figure~\ref{subfig:case1} is incident with only one hexagon of $\mathcal{K}$;
the one in Figure~\ref{subfig:case2} is incident with 2 hexagons of $\mathcal{K}$. We
consider the subgraph containing boundary vertices and edges. In the first case vertex
$v$ is incident with edges $e$ and $e'$ and has degree 2. In the second case, $v$ is
incident with $e$ and $e''$ and also has degree 2. Therefore, the subgraph containing
boundary vertices and edges is a 2-regular graph which is a union of vertex-disjoint cycles.
\begin{figure}[!ht]
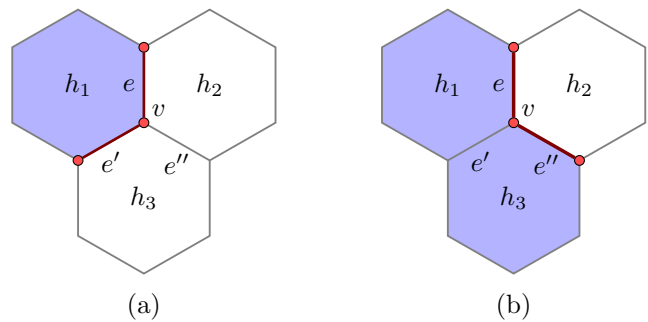

\centering
\subfigure[]{
\label{subfig:case1}
\input{img/border1.tikz}
}
\qquad
\subfigure[]{
\label{subfig:case2}
\input{img/border2.tikz}
}
\caption{Two types of boundary vertices. Shaded hexagons are present in the
coronoid. Vertex $v$ is of degree 2 in the first case and of degree 3 in the second case.
Boundary vertices and edges are emphasized.} 
\label{fig:typesofboundaryvert}
\end{figure}
By the Jordan curve theorem \cite{thomassen1992}, every cycle splits the plane into two
disconnected regions. Consider a hexagon $h \in \mathcal{K}^\complement$. It is contained
in a region that is surrounded by one of the cycles. No member of $\mathcal{K}$ is inside
this region, because a pathway of hexagons cannot enter the region elsewhere but through
the perimeter. Therefore, perimeters separate corona holes from the coronoid $\mathcal{K}$.
By Theorem~\ref{thm:decompositionthm}, there must be the same number of cycles as there
are corona holes (including the outer face).
\end{proof}

\noindent For two graphs $G$ and $H$ by $H \hookrightarrow G$ we denote embedding (injective homomorphism)
of graph $H$ into $G$. Note that graph $G(\mathcal{K})$ is not a plane graph and does not
possess any geometric information. Let $\mathcal{A}$ be anthracene. From
Figure~\ref{fig:drawinganthracene}, it is clear that $G(\mathcal{A})$ can be drawn in the
plane in many different ways.

\begin{figure}[!h]
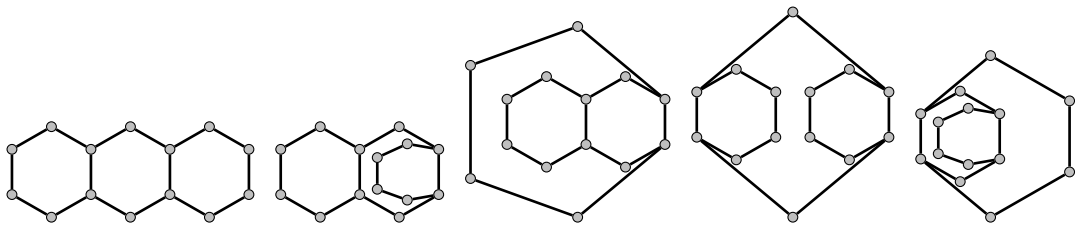

\centering
\input{img/anthracene.tikz}
\input{img/anthracene2.tikz}
\input{img/anthracene3.tikz}
\input{img/anthracene4.tikz}
\input{img/anthracene5.tikz}
\caption{Different drawings of the anthracene graph $G(\mathcal{A})$ in the Euclidean plane.}
\label{fig:drawinganthracene}
\end{figure}

\noindent The most appropriate drawing of anthracene for chemical purposes is the left-most. It is the
only one that can be obtained from the hexagonal grid $\mathcal{H}$. There is an embedding
of a coronoid graph into the hexagonal grid $\mathcal{H}$ called the \emph{natural embedding}.
Recall that $G(\mathcal{K})$ was obtained from $\mathcal{K}$ by taking all 0-cells and
1-cells of corresponding subcomplex. The natural embedding just sends the graph back to
its place of birth. We have the following theorem to tell us that there is only one drawing
up to symmetries of the hexagonal grid: 

\begin{theorem}
\label{thm:uniqueembedding}
Let $\mathcal{K}$ be a coronoid and let $C \subseteq G(\mathcal{K})$ be a cycle of length 6
($|C| = 6$). Then $C \hookrightarrow G(\mathcal{H})$ can be extended to
$G(\mathcal{K}) \hookrightarrow G(\mathcal{H})$ in an unique way.
\end{theorem}

\begin{proof}
Let $e = uv \in E(C_6)$. Then $e \hookrightarrow G(\mathcal{H})$ can be extended to
$C_6 \hookrightarrow G(\mathcal{H})$ in two different ways.

The distance between hexagons $h$ and $k$ in $\mathcal{K}$, $d(h, k)$, is the smallest
$n$ for which a sequence $h_0=h, h_1, \ldots, h_n = k$ of sequentially adjacent hexagons
of $\mathcal{K}$ exist. Let $\mathcal{K} = \{h_1, h_2, \ldots, h_d\}$. Suppose that
$C = \partial h_1$ and that $d(h_1, h_i) \leq d(h_1, d_{i+1})$. We start with
$C \hookrightarrow G(\mathcal{H})$ and extend it step by step. On $i$-th step we find images
of those vertices of $\partial h_i$ that are not already embedded. For $h_i$ there exists
some $h_j$, $j < i$, such that $\partial h_j$ was already embedded and $h_i$ shares an edge
$e_i$ with $h_j$. If vertices of $\partial h_i$ were already embedded, there is nothing
left to do. Otherwise, $e_i \hookrightarrow G(\mathcal{H})$ can be extended to
$\partial h_i \hookrightarrow G(\mathcal{H})$ in two ways. As we are constructing an injective
homomorphism, images of $\partial h_j$ and $\partial h_i$ may not overlap. Therefore,
only one option remains. 
\end{proof}

\noindent This constructive proof gives rise to an algorithm for embedding an arbitrary coronoid
graph $G$ into the hexagonal lattice. If anything goes wrong during this procedure, that
means that the input graph $G$ was not a valid coronoid graph. If the graph $G$ is a
valid coronoid graph, the algorithm will always finish successfully.

Theorem~\ref{thm:uniqueembedding} no longer holds, if we permit arbitrary
subgraphs of $G(\mathcal{H})$. (See the examples given in Figure~\ref{fig:notunique},
Figure~\ref{fig:bad_embedding} and Figure~\ref{fig:bad_embedding2}.)

\begin{figure}[!h]
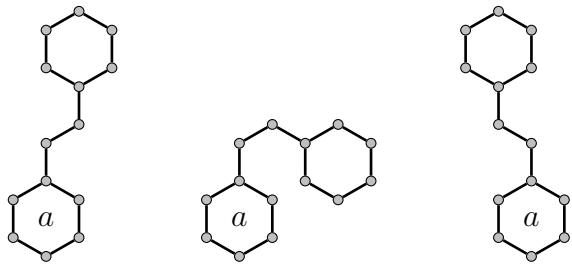

\centering
\input{img/not_unique1.tikz} \qquad
\input{img/not_unique2.tikz} \qquad
\input{img/not_unique3.tikz}
\caption{A graph that consists of two hexagons that are connected by a path of length 3
can be embedded in the hexagonal lattice in more than one way, even when the hexagon
denoted by $a$ is fixed.}
\label{fig:notunique}
\end{figure}

\begin{figure}[!ht]
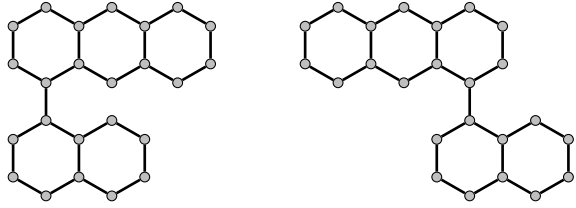

\centering
\input{img/another_not_unique1.tikz}
\qquad
\input{img/another_not_unique2.tikz}
\caption{Another example of a non-coronoid subgraph of the lattice that can be
embedded in two different ways.}
\label{fig:bad_embedding}
\end{figure}

\begin{figure}[!ht]
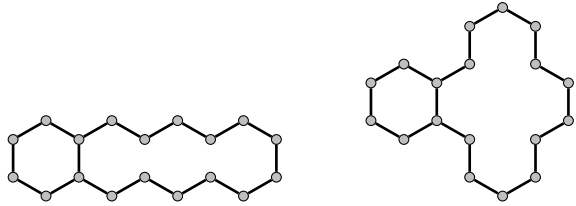

\centering
\input{img/2conn_notuniq.tikz} \qquad
\input{img/2conn_notuniq2.tikz}
\caption{An example of a 2-connected non-coronoid graph that can be embedded in more than
one way. (Not all embeddings are listed.)}
\label{fig:bad_embedding2}
\end{figure}

\begin{lemma}
\label{lemma:thereisanondeg}
Let $\mathcal{K}$ be a coronoid and let $G \coloneqq G(\mathcal{K})$. Then there exists
(up to symmetry of $\mathcal{H}$) exactly one non-degenerate coronoid $\mathcal{N}$ such
that $G = G(\mathcal{N})$. Moreover, $\mathcal{N} \cong \mathit{NonDeg}(\mathcal{K})$.
\end{lemma}

\begin{proof}
Choose some hexagon $h \in \mathcal{H}$ and choose an arbitrary 6-cycle $C$ in $G$.
Let $\phi \colon C \hookrightarrow G(\mathcal{H})$ be an embedding such that
$\phi(C) = \partial h$. By Theorem~\ref{thm:uniqueembedding}, $\phi$ can be
extended to $\Phi \colon G \hookrightarrow G(\mathcal{H})$ in a unique way. Let
$\mathcal{N} = \{ h \in \mathcal{H} \mid \partial h \subseteq \Phi(G) \}$. For
every 6-cycle $C$ in $\Phi(G)$ there exists some $h_C \in \mathcal{N}$ such that
$\Phi(C) = \partial h_C$.

First we show that $\mathcal{N}$ is non-degenerate. Suppose that it has a corona hole
of size 1, i.e., there exists $\widetilde{h} \in \mathcal{N}^\complement$ which is
sorrounded by $h_1, h_2, \ldots, h_6 \in \mathcal{N}$. Every edge of cycle
$\partial \widetilde{h}$ is present in $\Phi(G)$ because all hexagons $h_1, \ldots, h_6$
are present in $\mathcal{N}$. But then, by definition of $\mathcal{N}$, also
$\widetilde{h} \in \mathcal{N}$. A contradiction.

Now, we show that $\mathcal{N} \cong \mathit{NonDeg}(\mathcal{K})$. Without loss of
generality, we can assume that $\Phi$ is the natural embedding. If $k \in \mathcal{K}$
then from the definition of $G(\mathcal{K})$ we conclude that $\partial k \in G(\mathcal{K})$.
Therefore, $k \in \mathcal{N}$. This means that $\mathcal{K} \subseteq \mathcal{N}$.

Let $h \in \mathcal{N} \setminus \mathcal{K}$. Let $\tilde{h} \in N(h)$ and suppose
that $\tilde{h} \notin \mathcal{K}$. We know that $h$ and $\tilde{h}$ share an edge $e$
which does not belong to $G(\mathcal{K})$, because neither $h$ nor $\tilde{h}$ belongs
to $\mathcal{K}$. But $e \in G(\mathcal{K})$ by definition of $\mathcal{N}$. This is a
contradiction. Therefore, $\tilde{h} \in \mathcal{K}$ for every $\tilde{h} \in N(h)$.
With other words, $h$ is a degenerate corona hole inside $\mathcal{K}$. That means
that $\mathcal{N}$ is obtained from $\mathcal{K}$ by adding degenerate corona holes.
\end{proof}

\noindent We use the terminology of Gutman and Cyvin \cite{gutman1989}. Each boundary cycle is
called a \emph{perimeter.} In a proper coronoid there is one \emph{outer perimeter} and
one or more \emph{inner perimeters}. It may happen that one of the inner perimeters is
longer than the outer perimeter (see Figure \ref{fig:long_border}).
\begin{figure}[!h]
\centering
\input{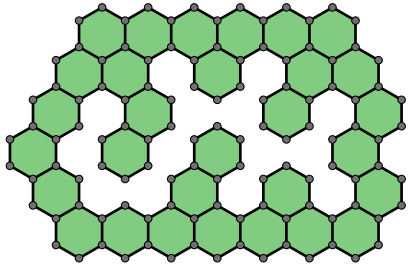}
\caption{The outer perimeter of this single coronoid is of length 48 whilst the inner
perimeter is of length 58.}
\label{fig:long_border}
\end{figure}
Nevertheless, the cycle of a coronoid graph that belong to the outer perimeter can be
easily recognised. By Theorem~\ref{thm:uniqueembedding}, a coronoid can be embedded in
the hexagonal lattice in a unique way (up to symmetry). Then the left-most vertex of
the coronoid belongs to the outer perimeter.

A \emph{corona hole} is a unique benzenoid that can fill the interior of an inner perimeter.
An inner perimeter of a coronoid may be viewed as the (outer) perimeter of the corresponding
corona hole. The roles of boundary vertices of degrees 2 and 3 are interchanged when we make
this change of viewpoint. A corona hole of a non-degenerate coronoid has at least 2 hexagons
and the corresponding inner perimeter has at least two vertices of degree 2. A typical
coronoid is schematically illustrated in Figure~\ref{fig:fig3}.

\begin{figure}[!h]
\centering
\input{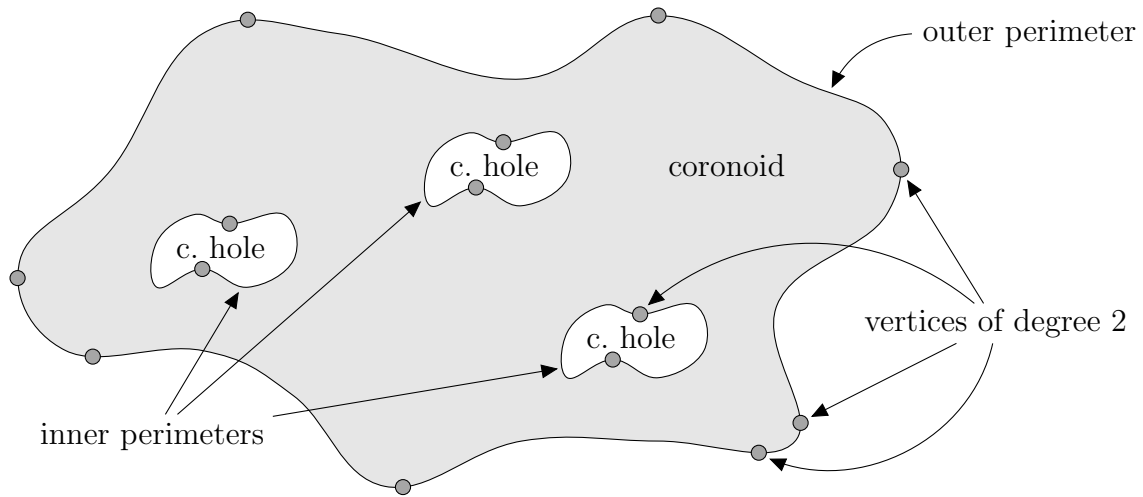}
\caption{Schematic illustration of a typical coronoid. The outer perimeter has at least 6
vertices of degree 2 and any inner perimeter has at least 2 vertices of degree 2.}
\label{fig:fig3}
\end{figure}

\noindent There are exactly two vertices of degree 2 on the inner perimeter in the case of a naphthalene
corona hole. For all other corona holes the number of such vertices is strictly greater than 2.
A lower bound can easily be obtained from the equations in \cite{gutman1989}.

\begin{proposition}
Let $h$, $n$, $m$ and $n_i$ denote, respectively, the number of hexagons, vertices,
edges and internal vertices in a corona hole. The number of boundary vertices of degree 3,
which correspond to vertices of degree 2 on the inner perimeter, is
$$
2h - 2 - n_i = n - 2h - 4 \geq \sqrt{12h - 3} - 3.
$$
\end{proposition}

\begin{proof}
Let $\nu$ denote the number of boundary vertices of degree 3 that belong to the corona hole
(those vertices are exactly degree-2 vertices of the corresponding inner perimeter).
In \cite{gutman1989} one can find the following equation:
\begin{equation}
\label{eqn:benzorel1}
\nu = 2h - 2 - n_i.
\end{equation}
Using $n = 4h + 2 - n_i$ \cite{gutman1989}, we can eliminate variable $n_i$ from
(\ref{eqn:benzorel1}) to obtain $\nu = n - 2h - 4$. Then apply $2h + 1 + \sqrt{12h - 3} \leq n$
\cite{gutman1989} to get $\nu \geq \sqrt{12h - 3} - 3$.
\end{proof}

\section{Patches and Perforated Patches}
First we generalise benzenoids and coronoids to patches and perforated patches, respectively.
Essentially, a \emph{patch} is a (2-connected) plane graph similar to a benzenoid in which various
polygons may be used instead of hexagons alone. All internal vertices are of degree 3, whilst
boundary vertices are of degree 2 or 3 (see example in Figure \ref{fig:patchexample}).
\begin{figure}[!h]
\centering
\input{img/patch.tikz}
\caption{A patch.}
\label{fig:patchexample}
\end{figure}

We will present a mathematical formalisation which is based on the treatment of coronoids
and benzenoids in the previous section of this paper. As we will see later, our definition
of a \emph{fullerene patch} is compatible with Graver's definition
\cite{gravergraves2010,gravergravesgraves2014}. There is also a notion of a
$(m, k)$-patch which received a lot of attention in the past years
\cite{brinkmann2005,brinkmann2009,graver2003,guo2002}. By our definition, faces may have
a range of different degrees, but $m = 3$.

Our point of departure is a finite plane cubic (simple) graph $G$, which divides the
Euclidean plane $\mathbb{R}^2$ into several regions called faces. The collection of all faces
is denoted $\mathcal{F}_G$. One face is unbounded and the rest are bounded. Two examples of
plane cubic graphs are in Figure~\ref{fig:examplesOfPlaneGraphs}.
\begin{figure}[!h]
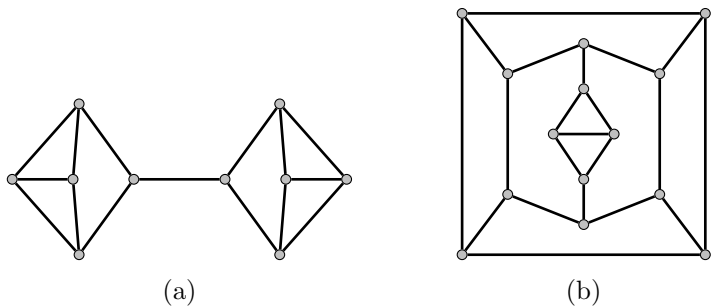

\centering
\subfigure[]{
\label{sfig:toucesitself}
\input{img/cubic01.tikz}
}
\qquad
\subfigure[]{
\label{sfig:touchesmore}
\input{img/cubic02.tikz}
}
\caption{Two examples of cubic plane graphs.}
\label{fig:examplesOfPlaneGraphs}
\end{figure}
Note that in the previous section, the role of graph $G$ was taken by an infinite cubic graph
with all faces hexagons which we called hexagonal lattice. Here we restrict our attention to
finite graphs $G$. Later on, we will compare finite and infinite versions of this theory.

The hexagonal lattice has additional nice properties. First of all, no two faces of $\mathcal{H}$
share more than one edge; the graph in Figure~\ref{sfig:touchesmore} does not have that
property. Also, no face of $\mathcal{H}$ is incident with itself; the graph in
Figure~\ref{sfig:toucesitself} does not have that property, since one if its edges is incident
to only one face (the outer face). Here, we also permit graphs which are not 2-connected,
such as the one on Figure~\ref{sfig:toucesitself}.

Let $\mathcal{P} \subseteq \mathcal{F}_G$ be some arbitrary subcollection of faces and let
$a, b \in \mathcal{P}$. We say that $a$ and $b$ are \emph{adjacent}, $a \sim b$, if they share an
edge. We define relation $\equiv_\mathcal{P}$ in exact same way as before, i.e.,\
$a \equiv_{\mathcal{P}} b$ if there is a sequence $c_0 = a, c_1, c_2, \ldots, c_m = b$ such
that $c_{i-1} \sim c_{i}$ and $c_i \in \mathcal{P}$. $\mathcal{P}$ is \emph{connected} if
$a \equiv_{\mathcal{P}} b$ for any $a, b \in \mathcal{P}$. Set $\mathcal{P}$ is naturally
decomposed into \emph{connected components}.

\begin{definition}
Let $G$ be any finite plane cubic graph. A proper subset $\mathcal{P} \subset \mathcal{F}_G$
is called a \emph{perforated patch} if it is connected.
\end{definition}

\begin{definition}
Let $G$ be any finite plane cubic graph. A proper subset $\mathcal{P} \subset \mathcal{F}_G$
is called a \emph{patch} if $\mathcal{P}$ is connected and
$\mathcal{P}^\complement = \mathcal{F}_G \setminus \mathcal{P}$ is also connected.
\end{definition}

\noindent Observe that finiteness of set $\mathcal{K}$ in Definitions~\ref{def:coronoid} and \ref{def:benzenoid}
implies that $\mathcal{K}$ is a proper subset of $\mathcal{H}$. In this finite version of those
two definitions we had to make that explicit.

Most observations of the previous section about coronoids and benzenoids are also true for their
corresponding generalisations, namely perforated patches and patches. In some cases the proof
remains essentially the same and in some cases it slightly simplifies, since we do not have to
deal with infinity anymore. We will transcribe results of the previous section into this new
language and omit the proofs. When appropriate, we will give some clarifications. For our
considerations, only the combinatorial data is relevant (adjacency between faces and the cyclic
ordering of edges incident to a common vertex). Geometric details of the drawing are unimportant.
Moreover, the outer face does not play a special role. The graph $G$ could also be embedded on
the sphere $S^2$, where all faces would be bounded. All the combinatorial information would,
of course, remain exactly the same.

\begin{lemma}
Let $G$ be any finite plane cubic graph and let $\emptyset \neq \mathcal{P} \subseteq \mathcal{F}_G$.
Let $a \in \mathcal{P}$ be any of its faces. Then $a$ belongs to some connected component
$\mathcal{C}_i$ of $\mathcal{P}$. Let $b \in N(a)$. Then either $b \in \mathcal{C}_i$
or $b \in \mathcal{P}^\complement$.
\hfill $\square$
\end{lemma}

\begin{lemma}
Let $G$ be any finite plane cubic graph and let $\emptyset \neq \mathcal{P} \subset \mathcal{F}_G$.
Let $\mathcal{C}_a$ be any connected component of its complement $\mathcal{P}^\complement$.
Then there exists a face $\tilde{a} \in \mathcal{C}_a$ that is adjacent to some face
in $\mathcal{P}$.
\hfill $\square$
\end{lemma}

\begin{lemma}
\label{lemma:pathwaylemmas}
Let $G$ be any finite plane cubic graph and let $\mathcal{P}$ be a perforated patch. Let
$\mathcal{C}_a$ be a connected component of its complement $\mathcal{P}^\complement$,
$a \in \mathcal{C}_a$ and $p \in \mathcal{P}$. Then $a \equiv_{\mathcal{P} \cup \mathcal{C}_a} p$.
Let $\mathcal{C}_b$ be another connected complement of $\mathcal{P}^\complement$. Then
$a \equiv_{\mathcal{P} \cup \mathcal{C}_a \cup \mathcal{C}_b} b$.
\hfill $\square$
\end{lemma}

The above Lemma~\ref{lemma:pathwaylemmas} corresponds to both Lemma~\ref{lemma:pathlemma} and
Corollary~\ref{cor:pathcor}.

\begin{lemma}
\label{lemma:reallytriviallemma}
Let $G$ be any finite plane cubic graph and let $\emptyset \neq \mathcal{P} \subseteq \mathcal{F}_G$.
Then the complement $\mathcal{P}^\complement$ of $\mathcal{P}$ consists of finitely many connected components:
$$
\mathcal{P}^\complement = \mathcal{C}_1 \sqcup \mathcal{C}_2 \sqcup \cdots \sqcup \mathcal{C}_d.
$$
Each of the components $\mathcal{C}_i$, $i \geq 1$, is a perforated patch. If $\mathcal{P}$
is a perforated patch, then each $\mathcal{C}_i$ is a patch.
\hfill $\square$
\end{lemma}

\noindent The above Lemma~\ref{lemma:reallytriviallemma} corresponds to both Lemma~\ref{lemma:decomposition}
and Theorem~\ref{thm:decompositionthm}. Since we do not have to deal with an infinite
number of faces, the proof becomes trivial. The definition of the benzenoid closure
was natural. Here, one should be slightly more careful:

\begin{definition}
Let $\mathcal{P}$ be a perforated patch. The \emph{closure} of $\mathcal{P}$ with
respect to $p \in \mathcal{P}^\complement$, denoted $\mathit{Cl}(\mathcal{P}, p)$ is
the intersection of all those patches which include $\mathcal{P}$ as a subset and do
not contain $p$ among their faces, i.e.,
$$
\mathit{Cl}(\mathcal{P}, p) = \bigcap \{ \mathcal{Q} \mid \mathcal{Q}\text{ is patch} \land
\mathcal{P} \subseteq \mathcal{Q} \land p \notin \mathcal{Q} \}.
$$
\end{definition}

\noindent Let us investigate what happens if the extra condition, i.e.,\ exclusion of a designated
face $p$, is omitted. By Lemma~\ref{lemma:reallytriviallemma},
$\mathcal{P}^\complement = \mathcal{C}_1 \sqcup \mathcal{C}_2 \sqcup \cdots \sqcup \mathcal{C}_d$
where each $\mathcal{C}_i$ is a patch. Define
$\mathcal{P}_i \coloneqq \mathcal{P} \sqcup \mathcal{C}_1 \sqcup \ldots \sqcup \mathcal{C}_{i-1}
\sqcup \mathcal{C}_{i+1} \sqcup \cdots \sqcup \mathcal{C}_d$. It is easy to see that each
$\mathcal{P}_i$, $1 \leq i \leq d$, is a patch. Clearly, $\mathcal{P} \subseteq \mathcal{P}_i$
for each $i = 1, \ldots, d$. But $\bigcap\, \{\mathcal{P}_i \mid 1 \leq i \leq d \} = \mathcal{P}$.
Without that extra condition the definition would not make sense. In most cases this
``forbidden'' face $p$ can be chosen in advance and one can deal with only those perforated
patches which do not include face $p$. Then we can write $\mathit{Cl}(\mathcal{P})$ instead
of $\mathit{Cl}(\mathcal{P}, p)$ without introducing any ambiguity. The most natural candidate
for the forbidden face is, of course, the outer face. Let us denote
$\mathit{Pat}(\mathcal{P}, p) = \{ \mathcal{Q} \mid \mathcal{Q}\text{ is patch} \land
\mathcal{P} \subseteq \mathcal{Q} \land p \notin \mathcal{Q} \}$ for convenience.
Lemma~\ref{lemma:closure} and Proposition~\ref{prop:propertiesofcl} give rise to the
following analogue in the theory of perforated patches:

\begin{lemma}
Let $\mathcal{P}$ be a perforated patch and $p \in \mathcal{P}^\complement$ a face in its
complement. Let
$\mathcal{P}^\complement = \mathcal{C}_1 \sqcup \mathcal{C}_2 \sqcup \cdots \sqcup \mathcal{C}_d$
as in Lemma~\ref{lemma:reallytriviallemma}. Then the closure of $\mathcal{P}$ with respect to $p$ is
$$
\mathit{Cl}(\mathcal{P}, p) = \mathcal{P} \sqcup \mathcal{C}_1 \sqcup \cdots \sqcup
\mathcal{C}_{j-1} \sqcup \mathcal{C}_{j+1} \sqcup \cdots \sqcup \mathcal{C}_d,
$$
where $\mathcal{P}_j$ is the connected component of $\mathcal{P}^\complement$ which
includes $p$ among its faces, i.e.,\ $p \in \mathcal{C}_j$. Moreover, the closure with
respect to $p$ is an operation on the set of perforated patches without $p$ that
satisfies the following conditions:
\begin{enumerate}
\renewcommand{\labelenumi}{(\alph{enumi})}
\item $\mathcal{P} \subseteq \mathit{Cl}(\mathcal{P}, p)$,
\item $\mathcal{P} \subseteq \mathcal{Q} \implies \mathit{Cl}(\mathcal{P}, p)
\subseteq \mathit{Cl}(\mathcal{Q}, p)$, and
\item $\mathit{Cl}(\mathit{Cl}(\mathcal{P}, p), p) = \mathit{Cl}(\mathcal{P}, p)$.
\end{enumerate}
\hfill $\square$
\end{lemma}

\noindent The following lemma is an analogue to Definition~\ref{def:alternativeclosure} and
Lemma~\ref{lem:alternativeclop} from the theory of coronoid hydrocarbons:

\begin{lemma}
Let $\mathcal{P}$ be a perforated patch such that $p \notin \mathcal{P}$. Let
$\widetilde{\mathcal{P}} \subseteq \mathcal{F}_G$ satisfy the following conditions:
\begin{enumerate}
\renewcommand{\labelenumi}{(\roman{enumi})}
\item $\widetilde{\mathcal{P}}$ is a patch without $p$,
\item $\mathcal{P} \subseteq \widetilde{\mathcal{P}}$, and
\item $\mathcal{P} \subseteq \mathcal{Q} \implies \widetilde{\mathcal{P}} \subseteq
\mathcal{Q}$ for every patch $\mathcal{Q}$ without $p$.
\end{enumerate}
There exists a unique $\widetilde{\mathcal{P}}$ with given properties and
$\widetilde{\mathcal{P}} = \mathit{Cl}(\mathcal{P}, p)$.
\hfill $\square$
\end{lemma}

\begin{figure}[!ht]
\centering
\input{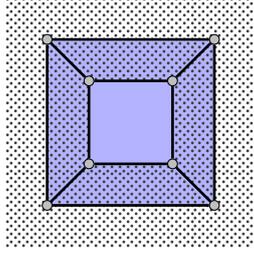}
\caption{The intersection of two patches may be a perforated patch.}
\label{fig:intersectofpatches}
\end{figure}

\noindent Some caution should be used in making the analogy with Lemma~\ref{lemma:intersectionofbenzenoids}:
as seen in Figure~\ref{fig:intersectofpatches}, the intersection of two patches may be
a perforated patch. However, the following is true:

\begin{lemma}
Let $\mathcal{P}$ and $\mathcal{Q}$ be two patches such that their complements share a face,
i.e., $\mathcal{P}^\complement \cap \mathcal{Q}^\complement \neq \emptyset$. Then the
intersection of patches $\mathcal{P}$ and $\mathcal{Q}$ is disjoint union of patches.
\hfill $\square$ 
\end{lemma}

\noindent To obtain a perforated patch we choose a set of faces of the plane cubic graph $G$ that
constitute a connected region. Sometimes we talk about ``removing faces''. This means
that we select members of the complement, i.e., faces that will not be present in the
perforated patch.

\noindent Given a perforated patch as a plane graph it is not possible to detect (in the general case)
which faces are corona holes and which not. An example is given in Figure~\ref{fig:cantdetect}.
Therefore, plane graphs do not give a sufficient model for perforated patches. One has to
indicate which faces are actually present, and which are not. This means that
Lemma~\ref{lemma:thereisanondeg} has no equivalent in this theory.

\begin{figure}[!ht]
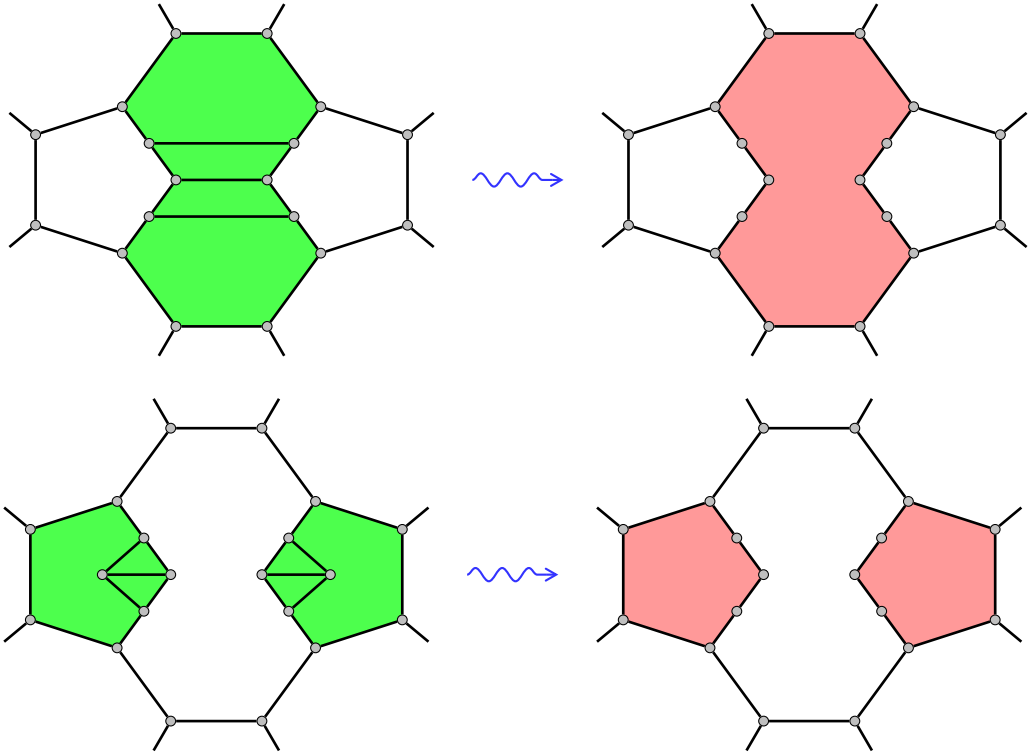

\centering
\input{img/ambigious01.tikz}
\\[0.5cm]
\input{img/ambigious02.tikz}
\caption{Two distinct perforated patches with the same skeleton. The faces that we removed,
i.e.,\ the faces in the complement, are indicated by shading 
and the corresponding hole is also shaded.} 
\label{fig:cantdetect}
\end{figure}

A patch is called \emph{$k$-connected} when its skeleton is $k$-connected. A face which is
incident to an edge from both sides is called an \emph{ill-behaved} face. The following proposition
precisely characterizes 2-connected patches:

\begin{proposition}
A patch $\mathcal{P}$ is 2-connected if and only if it contains no ill-behaved faces.
\end{proposition}

\begin{proof}
An ill-behaved face means there is a bridge (cut-edge) so it is not 2-connected.
If there are no ill-behaved faces: start with a single face. Its skeleton is a cycle which
is 2-connected. Then iteratively add adjacent faces. The newly added edges form one or more
paths that are connected to the existing graph. In this way we construct ear-decomposition.
\end{proof}

\noindent It is possible to generalise this theory for the case when the cubic planar graph $G$ is
infinite. The standard embedding of the hexagonal lattice is such an example. If one tries
to use an arbitrary infinite cubic graph, problems of topological nature may arise. Some
remarks must therefore be made. Hexagonal grid, as we will see, has ceirtain nice properties.
In the proof of Lemma~\ref{lemma:adjacencylemma} we use the fact that
$\cup_{n=0}^{\infty} \mathcal{P}_n = \mathcal{H}$ without proving it. This holds when
there exists a finite pathway of faces between any two faces of the plane graph. With other
words, the distance between every two vertices of the dual graph is finite. Hexagonal grid
is an example. If this was not the case, the proof of Lemma~\ref{lemma:adjacencylemma} would
fail. Another important building block of this theory is Lemma~\ref{lemma:decomposition}.
In its proof, when we claim that $\mathcal{P}^\complement$ is connected, we implicitly use
the Jordan curve theorem. This renowned theorem seems obvious at the first sight, but its
proof happens to be involved. Another sensitive part of the proof is when we claim that
$\mathcal{P}_{n}$ contains finitely many faces. A nifty topologist could construct such an
example where this would fail. Luckily, every bounded region of the infinite hexgonal grid
contains finitely many faces. This is another condition on the infinite plane graph $G$ that
has to be met.

\section{Altans, Generalised Altans and Iterated Altans}
We will start by making a small extension of the definition of an altan as presented previously
\cite{basic2015}. Let $G$ be a graph and let $C$, \emph{the perimeter}, be a cycle in $G$ having
$k \geq 2$ vertices of degree 2. Then $A(G, C)$ will be the altan as defined in \cite{basic2015}
with respect to the degree-2 vertices of $C$. Those edges which connect degree-3 vertices on the
new cycle with $C$ will be called \emph{spokes}. 

\begin{example}
See Figure~\ref{fig:fig5}.

\begin{figure}[!h]
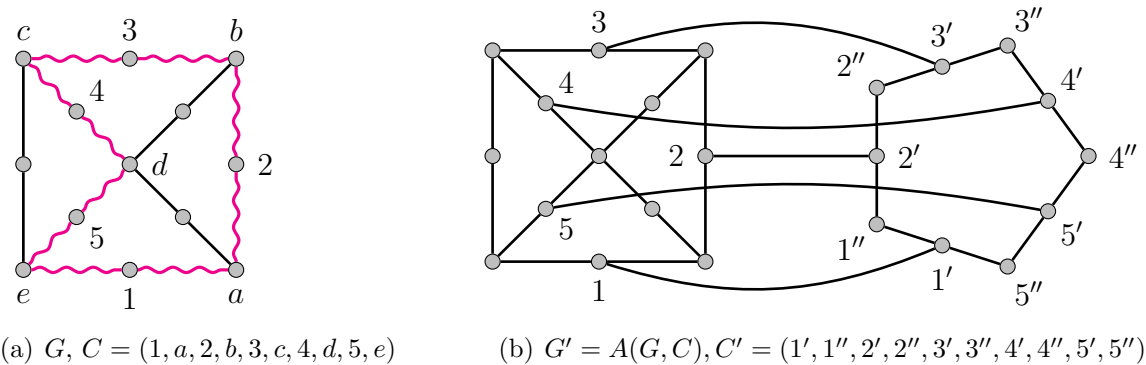

\centering
\subfigure[$G$, $C = (1, a, 2, b, 3, c, 4, d, 5, e)$]{
\input{img/figure_5.tikz} \quad\quad\quad
}
\qquad
\subfigure[$G' = A(G, C), C' = (1', 1'', 2', 2'', 3', 3'', 4', 4'', 5', 5'')$]{
\input{img/figure_5a.tikz}
}
\caption{A graph $G$ with a designated perimeter $C$ (on the left), and its altan $A(G, C)$ (on the right).}
\label{fig:fig5}
\end{figure}
\end{example}

A generalised altan is obtained by selecting a collection of cycles $C_1, C_2, \ldots, C_k$
in $G$ with the property that any degree-2 vertex of $G$ appears on at most one cycle
and that each of the cycles $C_i$ contains at least 2 vertices of degree 2. We call
$(G; C_1, \ldots, C_k)$ an \emph{admissible structure}. In addition we select a non-empty
subset of indices $J \subseteq \{1, 2, \ldots, k\}$ and perform the altan operation on all
cycles $C_j$, $j \in J$. We define
$$
A(G; C_1, \ldots, C_k; \emptyset) = (G; C_1, C_2, \ldots, C_k),
$$
and
$$
A(G; C_1, \ldots, C_k; J) = A(G'; C_1, \ldots, C_{j-1}, C_j', C_{j+1}, \ldots, C_k;  J \setminus j),
$$
for $J \neq \emptyset$, where $j = \min J$. Graph $G'$ is obtained from $G$ by adding a new
copy, denoted $C_{j}'$, of a cycle on $2d$ vertices, where $d$ is the number of degree-2
vertices on cycle $C_j$. Every second vertex on cycle $C_{j}'$ is attached to a degree-2
vertex on $C_{j}$. Usually we take either $|J| = 1$ or $|J| = k$. In the former case we are
dealing with the ordinary altan operation. In the latter case all perimeters are used. 

\begin{example}
\label{example:e34}
Let $G$ be the subdivided cube in Figure~\ref{fig:example34}.
Let $C_1 = (1, 2, 3, 4, 7, 6, 5)$, $C_2 = (1, 5, 9, 10, 14, 8)$ and $C_3 = (10, 11, 12, 13, 16, 15, 14)$.
\begin{figure}[!hb]
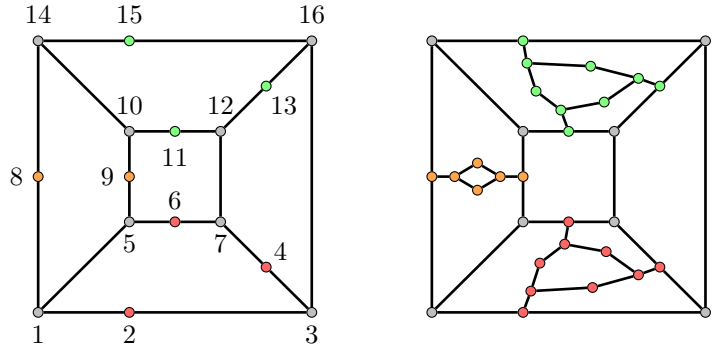

\centering
\input{img/example34.tikz}
\qquad
\input{img/example34b.tikz}
\caption{A subdivided cube (on left) and its generalised altan (on right).}
\label{fig:example34}
\end{figure}
Then $(G; C_1, C_2, C_3)$ is an admissible structure. Note that cycles $C_1, \ldots, C_k$
of an admissible structure need not be disjoint as long as no degree-2 vertex lies on
a shared part. The generalised altan $A(G;C_1, C_2, C_3; \{1, 2, 3\})$ is on the right
in Figure~\ref{fig:example34}.
\end{example}

We may apply the generalised altan operation iteratively. The order in which we apply individual
``local'' altan operations is irrelevant.

\subsection{Iterated generalised altans}
Let $(G; C_1, C_2, \ldots, C_k)$ be an admissible structure. Let
$\mathbf{n} = (n_1, n_2, \ldots, n_k)$ be an integer vector with $n_i \geq 0$.
Then
$$
A^\mathbf{n}(G; C_1, \ldots, C_k)
$$
denotes the generalised iterated altan. Let
$\mathbf{n^{\dagger}} = (n_1^{\dagger}, n_2^{\dagger}, \ldots, n_k^{\dagger})$ where
$$
n_i^{\dagger} = \begin{cases}
n_i - 1, & \text{if }n_i > 0\\
0, & \text{otherwise.}
\end{cases}
$$
Moreover, let $J_{\mathbf{n}}^\dagger = \{ i \mid n_i > 0 \}$. Then iterated generalised
altan can be defined in terms of the generalised altan in the following way
$$
A^\mathbf{n}(G; C_1, \ldots, C_k) = A^\mathbf{n^{\dagger}}(A(G; C_1, \ldots, C_k; J_{\mathbf{n}}^\dagger)).
$$
Naturally, $A^{\mathbf{0}}(G; C_1, \dots, C_k) = (G; C_1, \ldots, C_k)$ where
$\mathbf{0} = (0, 0, \ldots, 0)$.

\begin{example}
\label{example:e35}
Consider admissible structure $(G; C_1, C_2, C_3)$ from Example~\ref{example:e34}.
Iterated generalised altan $A^{(1, 0, 2)}(G; C_1, C_2, C_3)$ is shown in
Figure~\ref{fig:example35}.
\begin{figure}[!hb]
\centering
\input{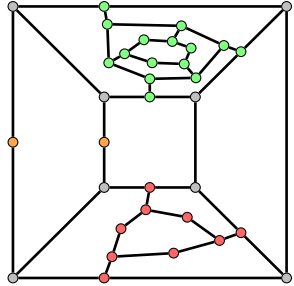}
\caption{The iterated generalised altan $A^{(1, 0, 2)}(G; C_1, C_2, C_3)$.}
\label{fig:example35}
\end{figure}
\end{example}

\subsection{Altans of coronoids and perforated patches}
From now on, by a coronoid we mean a non-degenerate coronoid. Each coronoid $\mathcal{K}$
with its perimeters is an admissible generalised altan structure. Hence,
$A^\mathbf{n}(\mathcal{K})$ is well-defined as soon as we label its perimeters. Note that
the cycles are exactly perimeters and they are disjoint. When $\mathbf{n} \neq \mathbf{0}$,
we call $A^{\mathbf{n}}(\mathcal{K})$ a \emph{proper generalised altan}.

\begin{example}
Let $\mathcal{K}$ be the coronoid in Figure~\ref{fig:example35} (the part consisting of
dotted 
faces). Let $C_1$ denote the left inner perimeter, $C_2$ the right inner perimeter and
$C_3$ the outer perimeter. Generalised altan $A^{(2, 3, 0)}(\mathcal{K})$ consists of
shaded 
and dotted 
faces in Figure~\ref{fig:example36}.
\begin{figure}[!htb]
\centering
\input{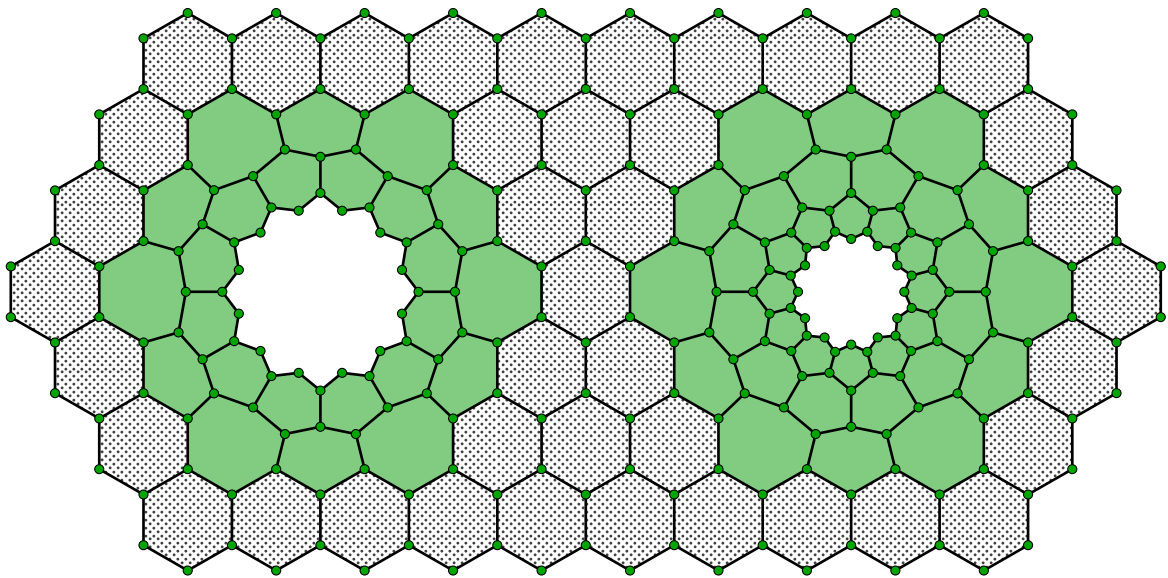}
\caption{$A^{(2, 3, 0)}(\mathcal{K})$.}
\label{fig:example36}
\end{figure}
\end{example}

\begin{proposition}
A proper generalised altan of a coronoid is not a coronoid.
\end{proposition}

\begin{proof}
In the case of a benzenoid, we restate Gutman's observation \cite{gutman2014} on the
structural features of altan-benzenoids. Every benzenoids contains a $(2, 2)$-edge,
i.e.,\ an edge connecting two degree-2 vertices. Figure~\ref{subfig:scenario1} shows a
fragment of a benzenoid with a $(2, 2)$-edge. This gives rise to a pentagon in its altan.
(The new vertices that are obtained by the altan operation are distinguished by shading.)

\begin{figure}[!ht]
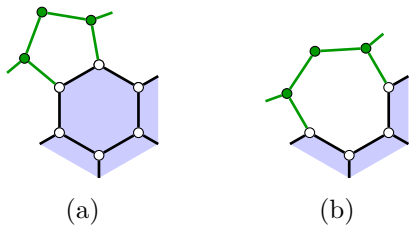

\centering
\subfigure[]{
\label{subfig:scenario1}
\input{img/pentagon_appears.tikz}
}
\qquad
\subfigure[]{
\label{subfig:scenario2}
\input{img/heptagon_appears.tikz}
}
\caption{The altan of a coronoid is no longer a coronoid.}
\label{fig:pentagon_appears}
\end{figure}

In the case of a coronoid, the same proof works for the outer perimeter. For an inner
perimeter, one can observe that it must contain at least one $(3, 3)$-edge which corresponds
to a $(2, 2)$-edge in its corona hole. This means that there are at least two vertices
of degree 3 between some pair of degree-2 vertices. This give rise to a heptagon or an
even larger face (see Figure~\ref{subfig:scenario2}).
\end{proof}

\noindent However, Gutman has shown \cite{gutman2014} that when the altan is performed on a convex
benzenoid \cite{cruzgutmanrada2012}, degrees of the newly obtained faces are limited to 5 and 6.
Traditionally, a patch is defined as a subcubic plane graph that has all its degree-2 vertices
on its outer perimeter. Clearly, the skeleton $G(\mathcal{P})$ of a patch $\mathcal{P}$ is such a graph.
The following result shows that our definition actually coincides with the traditional one:

\begin{proposition}
Let $G$ be a plane subcubic graph with all the degree-2 vertices on its outer perimeter.
Then there exists a plane cubic graph $\widetilde{G}$, such that $G \subseteq \widetilde{G}$,
all inner faces of $G$ are also faces of $\widetilde{G}$ and there exists patch
$\mathcal{P} \subseteq \mathcal{F}_{\widetilde{G}}$ such that $G = G(\mathcal{P})$. Moreover,
if $G$ is 2-connected with at least two degree-2 vertices, then there exists a 2-connected
graph $\widetilde{G}$.
\end{proposition}

\begin{proof}
If $G$ contains three or more degree-2 vertices, choose $C$ to be the outer perimeter of $G$
and make altan $A(G, C)$. Then remove all the newly obtained degree-2 vertices and reconnect its
neighbours (reverse operation to subdivision) as shown in Figure~\ref{fig:traditionalproof1}.
\begin{figure}[!h]
\centering
\input{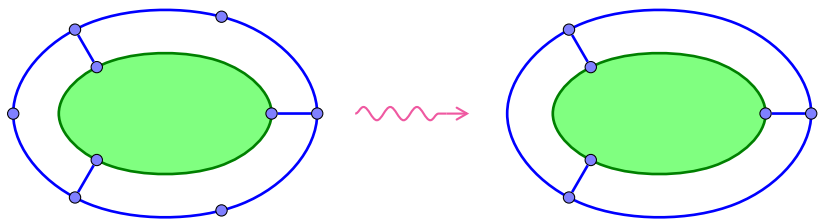}
\caption{Obtaining $\widetilde{G}$ from $G$ with three or more degree-2 vertices.}
\label{fig:traditionalproof1}
\end{figure}
It is trivial to verify that this graph is indeed the desired $\widetilde{G}$. If $G$ was
2-connected then its ear-decomposition can easily be extended to include the newly obtained
edges. This shows that $\widetilde{G}$ is also 2-connected.

If $G$ had only two degree-2 vertices, then the above procedure would yield a multigraph. It can
be fixed by subdiving its edges on the new outer perimeter to obtain 3 or more degree-2 vertices
and repeating this operation as shown in Figure~\ref{fig:traditionalproof2}.
\begin{figure}[!ht]
\centering
\input{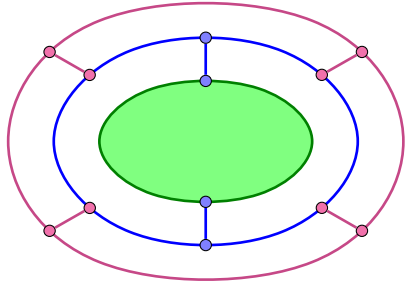}
\caption{Obtaining $\widetilde{G}$ from $G$ with two degree-2 vertices.}
\label{fig:traditionalproof2}
\end{figure}
Again, it is not hard to see that this yields the desired $\widetilde{G}$ and that the
ear-decomposition of $G$ can be extended to $\widetilde{G}$.

\begin{figure}[!h]
\centering
\input{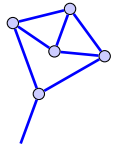}
\caption{A ``cherry''.}
\label{fig:traditionalproof3}
\end{figure}
If $G$ has only one degree-2 vertex, there is no hope of obtaining 2-connected $\widetilde{G}$.
Graph $\widetilde{G}$ can be obtain from $G$ by attaching a ``cherry'' (see Figure~\ref{fig:traditionalproof3})
to its degree 2 vertex.
\end{proof}

\noindent We can say more:

\begin{proposition}
Let $G$ be a plane subcubic \emph{bipartite} graph with all degree-2 vertices on its outer perimeter.
Then there exists a plane cubic \emph{bipartite} graph $\widetilde{G}$, such that
$G \subseteq \widetilde{G}$, all faces of $G$ are also faces of $\widetilde{G}$ and there exists
a patch $\mathcal{P} \subseteq \mathcal{F}_{\widetilde{G}}$ such that $G = G(\mathcal{P})$. Moreover,
if $G$ was 2-connected then there exists a 2-connected graph $\widetilde{G}$ with desired properties.
\end{proposition}

\begin{proof}
Suppose there are at least two degree-2 vertices of the same colour, i.e., belong to the same
set of the bipartition, in graph $G$ and denote those two vertices by $u$ and $v$. Without loss
of generality, we can assume they are coloured black. Those two vertices can be choosen in such
way that there are no other black vertices between them when we traverse the perimeter from
$u$ to $v$ in the clockwise direction (see Figure~\ref{subfig:traditionala}).
\begin{figure}[!ht]
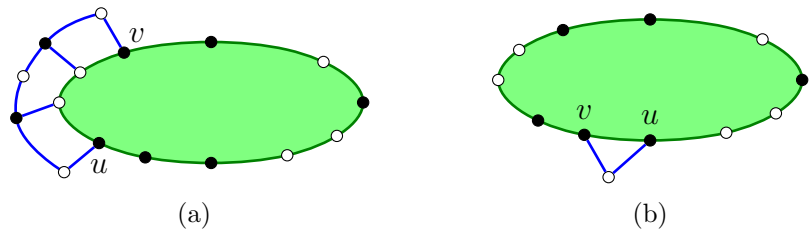

\centering
\subfigure[]{
\label{subfig:traditionala}
\input{img/traditional_proof1.tikz}
}
\qquad
\subfigure[]{
\label{subfig:traditionalb}
\input{img/traditional_proof7.tikz}
}
\caption{A step in obtaining bipartite $\widetilde{G}$ from bipartite $G$.}
\label{fig:traditionalproof4}
\end{figure}
However, there may be 0 or more white vertices on that path. Say there are one or more white vertices.
Label those vertices $w_1, w_2, \ldots, w_l$, where $l \geq 1$. Make an altan of $G$ with vertices
$(u, w_1, \ldots, w_l, v)$ as its \emph{peripheral root}. Label the newly obtained neighbours of
vertices $u$ and $v$ with $u'$ and $v'$, respectively. Then make a reverse subdivision operation which
removes all degree-2 vertices in the neighbourhood of $u'$ and $v'$. Also, remove the edge $u'v'$ to
obtain the graph in Figure~\ref{subfig:traditionala}. If there are no white vertices between $u$ and $v$,
add a new vertex to the graph and connect it to $u$ and $v$ as shown in Figure~\ref{subfig:traditionalb}.
In both cases, this graph is clearly bipartite and 2-connected if the graph $G$ was 2-connected. Also,
two black degree-2 vertices and $l$ white degree-2 vertices have disappered and $l+1$ new white
degree-2 vertices have emerged. The total number of degree-2 vertices is therefore decreased by one.

This procedure terminates when there are only two vertices left, which have to be of different
colours. (The situation with only one degree-2 vertex, say a white vertex, cannot occur.
The number of edges should be divisible by 3, because every black vertex has degree 3. On the
other hand, there is one white vertex of degree 2 and the rest have degree 3, which implies
that the number of edges is congruent 2 modulo 3, a contradiction.) 
\begin{figure}[!h]
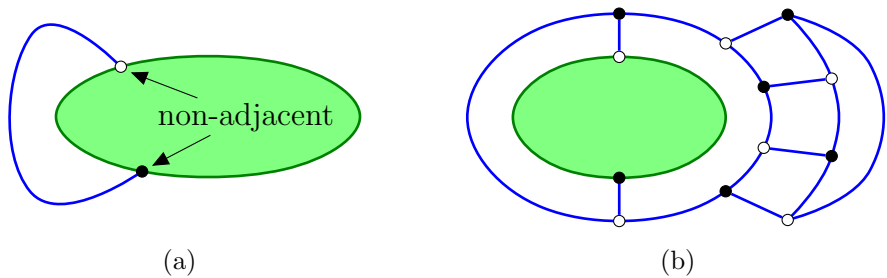

\centering
\subfigure[]{
\label{subfig:traditional2a}
\input{img/traditional_proof3.tikz}
}
\qquad
\subfigure[]{
\label{subfig:traditional2b}
\input{img/traditional_proof2.tikz}
}
\caption{The final step in obtaining bipartite $\widetilde{G}$ from bipartite $G$.}
\label{fig:traditionalproof5}
\end{figure}
If those final two vertices (which have to be of different colour) are non-adjacent,
connect them by an edge as shown in Figure~\ref{subfig:traditional2a}. If they are
adjacent, we would create a multigraph. In that case, use the construction shown in
Figure~\ref{subfig:traditional2b} to avoid the multigraph.
\end{proof}

\noindent A perforated patch with pentagonal and hexagonal faces is called a \emph{perforated fullerene patch}.
Similary, a patch with pentagonal and hexagonal faces is called a \emph{fullerene patch}.
One should be aware that the restriction to hexagonal and pentagonal faces applies within the patch
$\mathcal{P}$; other faces of the cubic graph $G$ from which the patch was derived may be of other sizes.
If such a graph $G$ with exclusively pentagonal and hexagonal faces exists, then the patch
(or perforated patch) can be extended to a fullerene. It is not easy to verify the existence of such $G$.

\begin{example}
Figure~\ref{fig:anythingcanhappen} shows that various possibilities can occur when we apply the
altan operation to a fullerene patch. In first case (left-hand side of Figure~\ref{fig:anythingcanhappen}),
the altan contains a 7-gon. In the second case (right-hand side), the altan is again a fullerene patch
\begin{figure}[!ht]
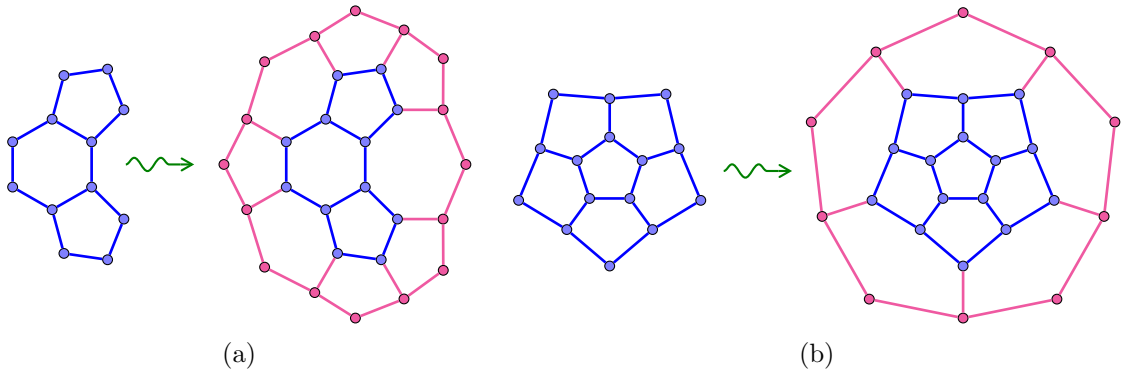

\centering
\subfigure[]{
\input{img/anything_happen.tikz}
}
\subfigure[]{
\input{img/anything_happen2.tikz}
}
\caption{Altans of fullerene patches.}
\label{fig:anythingcanhappen}
\end{figure}
(which may or may not extend to a fullerene).
\end{example}

\noindent There is another viewpoint we can take when dealing with (perforated) patches. In addition to the
skeleton $G(\mathcal{P})$, one can also obtain a planar \emph{pre-graph}, denoted $P(\mathcal{P})$.
It can be obtained from the plane graph $G$ by removing all vertices that are not incident to any
face of $\mathcal{P}$, together with all semiedges that are incident to removed vertices. In addition,
edges incident to two faces from $\mathcal{P}^\complement$ are also removed and replaced with two
half edges (as if the edge was cut in the middle). An example is given in Figure~\ref{fig:pregraphexample}.

\begin{figure}[!ht]
\centering
\input{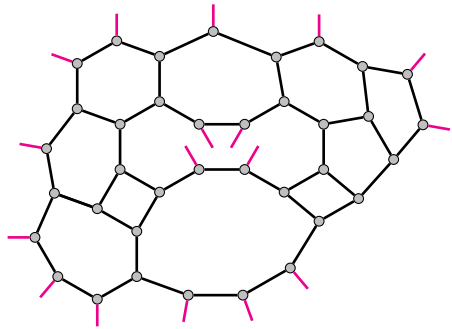}
\caption{A pre-graph of a patch. Half edges are ``without vertices'' on one end, i.e.,
they are ``dangling''.} 
\label{fig:pregraphexample}
\end{figure}

\begin{theorem}
\label{thm:altanofpatchispatch}
The generalised altan of a perforated patch $\mathcal{P}$ is a perforated patch. Moreover,
if $G(\mathcal{P})$ is 2-connected then $G(A^{\mathbf{n}}(\mathcal{P}))$ is also 2-connected.
\end{theorem}

\begin{proof}
The pre-graph of a perforated patch $\mathcal{P}$ is schematically illustrated in
Figure~\ref{fig:pregrapha}. Each hole corresponds to a void space that can be filled with an open
disc. There are also half edges (attached to degree-2 vertices of $G(\mathcal{P})$) which are drawn
inside those holes.
\begin{figure}[!ht]
\centering
\input{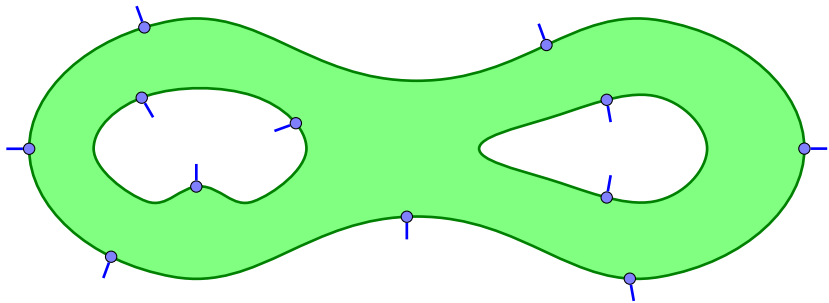}
\caption{Pre-graph of a perforated patch.}
\label{fig:pregrapha}
\end{figure}
When we perform an altan operation on that perforated patch, a cycle is drawn inside every hole on
which the altan operation is performed and an annulus of new faces (bounded by the new and the old
perimeter) is added to the patch. See Figure~\ref{fig:pregraphb} for an illustration. New holes have
the same number of half edges as they had before the operation. The parts that were removed from $G$
can be reattached to form the plane cubic graph. 
\begin{figure}[!ht]
\centering
\input{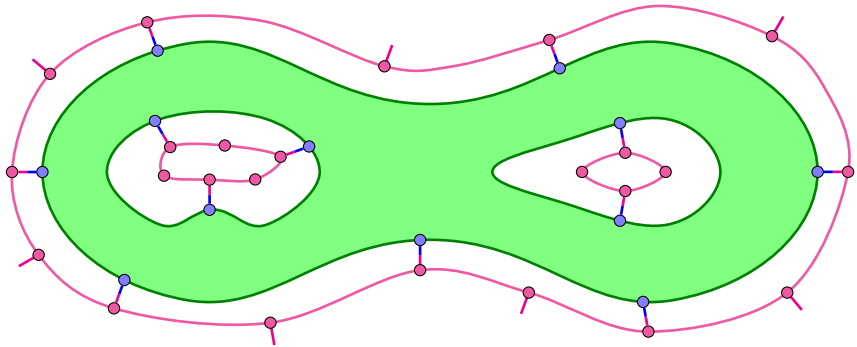}
\caption{Pre-graph of an altan of a perforated patch.}
\label{fig:pregraphb}
\end{figure}

It is clear that $G(A^\mathbf{n}(\mathcal{P}))$ is 2-connected if $G(\mathcal{P})$ is 2-connected.
The ear decomposition of $G(\mathcal{P})$ can easily be extended to include the newly obtained edges.
\end{proof}

\noindent The following corollary obviously follows from Theorem~\ref{thm:altanofpatchispatch}:

\begin{corollary}
The generalised altan of a patch $\mathcal{P}$ is a patch. Moreover, if $G(\mathcal{P})$ is 2-connected
then $G(A^{\mathbf{n}}(\mathcal{P}))$ is also 2-connected.
\end{corollary}

\section{Kekul\'e Structures and Pauling Bond Orders}
It is not hard to see that iteration of the generalised altan operation on coronoids and perforated
patches grows tubes on each perimeter, i.e.,\ we can visualize an embedded version in a way that is
reminiscent of the classic ruled surface of the graph in which some central planar perforated patch has
tubular towers growing out of it (in either up or down directions). See Figure~\ref{fig:towers_color}.

\begin{figure}[!h]
\centering
\input{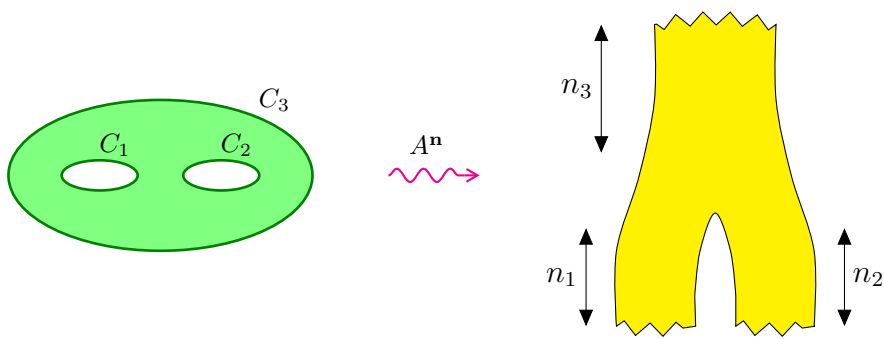}
\caption{Pants resulting from a disk with two holes by applying iterated altan operation. If the disk
has $K$ Kekul\'{e} structures then the pants have $K' = 2^{n_1 + n_2 + n_3} K$ Kekul\'{e} structures.}
\label{fig:fig7}
\end{figure}

The \emph{binary boundary code} for benzenoids is described in \cite{putz2011}. (This code is also known as PC-1 \cite{herndon1987}.) The binary boundary
code of a benzenoid is a sequence of degrees of consecutive vertices alongs its perimeter. Cyclic
shifts and reversal of the sequence are considered as equivalent codes. Traditionally, ones and zeroes
were used, but we will use 3s and 2s instead. To each perimeter $C_i$ of an admissible structure
$(G; C_1, \ldots, C_k)$ we will assign a binary boundary code, denoted $\mathit{BBC}(C_i)$.
Boundary-edges codes for benzenoids, introduced in \cite{hansen1996}, are useful on many occasions
\cite{kovic2014}, but in this case binary boundary codes are more natural.

\begin{example}
The first coronoid in Figure~\ref{fig:coroexamples} has three perimeters. Let $C_\infty$ denote the outer
perimeter and $C_1$ and $C_2$ the inner perimeters. Then
$$
\mathit{BBC}(C_\infty) = 323222332322332223232232323223232232
$$
and
$$
\mathit{BBC}(C_1) = \mathit{BBC}(C_2) = 3333233332.
$$
We will use $3^k$ as a short form of $\underbrace{33\ldots3}_{k}$. In this convention,
$\mathit{BBC}(C_1) = 3^423^42 = (3^42)^2$.
\end{example}

\begin{theorem}
Let $G$ be a coronoid and let $BBC(C) = 23^{\ell_1}23^{\ell_2}2\ldots 23^{\ell_d}$,
where $l_i \geq 0$ for $i = 1, \ldots, d$, be the binary boundary code for one of
its perimeters $C$. The degree of the $i$-th newly obtained face, $1 \leq i \leq d$,
is $\ell_i + 5$. Moreover, the binary boundary code of the new boundary is $(32)^d$.
\hfill $\square$
\end{theorem}

\begin{theorem}
Let $G$ be a perforated patch with $K$ Kekul\'e structures and let $G' = A^\mathbf{n}(G)$
be any of its generalised altans. Then the number of Kekul\'e structures in $G'$ is given
by $K' = 2^{|\mathbf{n}|} K$, where $|\mathbf{n}| = n_1 + n_2 + \cdots + n_k$. Furthermore:
\begin{enumerate}
\renewcommand{\labelenumi}{(\alph{enumi})}
\item No spoke belongs to a Kekulé structure.
\item If $n_i > 0$, all edges on the new perimeter belong to the same number,
$\frac{K'}{2}$, of Kekul\'e structures.
\end{enumerate}  \hfill $\square$
\end{theorem}

\noindent The following corollary follows straightfowardly from the above theorem:
\begin{corollary}
A generalised altan $A^\mathbf{n}(G)$ is Kekulean if and only if G is Kekulean.
\hfill $\square$
\end{corollary}

\begin{corollary}
Let $G$ be a perforated patch and let $G' = A^\mathbf{n}(G)$ be any of its
generalised altans. The Pauling Bond Order of the newly obtained edge $e$ is:
\begin{itemize}
\item 0 if $e$ is a spoke,
\item $\frac{1}{2}$ if $e$ is not a spoke.
\end{itemize}
Pauling Bond Orders of the edges that belong to the original graph $G$ remain the same.
\hfill $\square$
\end{corollary}
\noindent Note that graph $A^\mathbf{n}(G)$ was obtained from $G$ by adding new vertices
and edges. Therefore, $G$ is a subgraph of $A^\mathbf{n}(G)$ in a natural way.

\section{Conclusion}
Generalised altans are models of carbon nanostructures that are constructed by attachment
of carbon towers \cite{lukovits2003,sachs1996} to the holes in coronoid patches. Kekul\'{e} structures and Pauling Bond
Orders (and by implication ring currents \cite{fowlermyrvold2011,gomes1979,randic2010}) the nanostructure
can be derived in terms of those of the undecorated structure.

\begin{figure}[!hbt]
\centering
\includegraphics[scale=0.5]{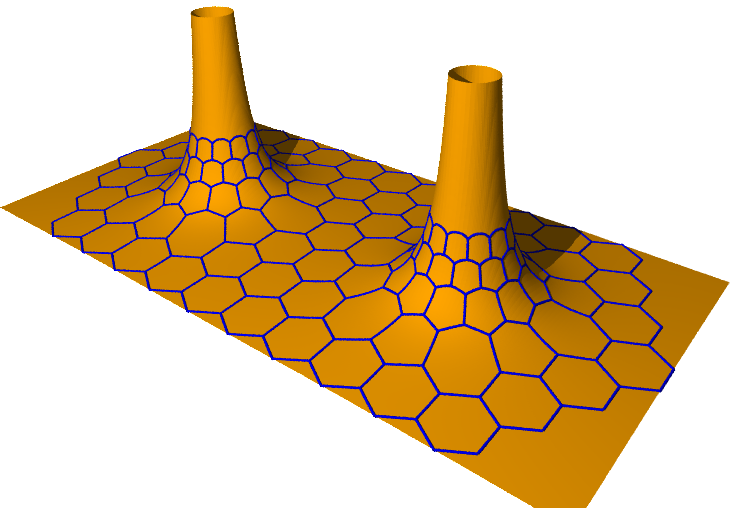}
\caption{Iteration of the altan construction leads to a carbon nanostructure in which nanotubes
grow out of the original holes of the coronoid. A given tube may grow up or down, as a  `chimney-stack'
or a `mine-shaft' on the graphene-like landscape, leading to isomeric structures that share
a common molecular graph.}
\label{fig:towers_color}
\end{figure}

\vspace{\baselineskip}
\noindent
{\bf Acknowledgements:} 
Research of T.\ P.\ and N.\ B.\ is supported in part by the ARRS Grant P1-0294. 

\bibliographystyle{amsplain-nodash}
\bibliography{coronoids_altans}

\providecommand{\bysame}{\leavevmode\hbox to3em{\hrulefill}\thinspace}
\providecommand{\MR}{\relax\ifhmode\unskip\space\fi MR }
\providecommand{\MRhref}[2]{%
  \href{http://www.ams.org/mathscinet-getitem?mr=#1}{#2}
}
\providecommand{\href}[2]{#2}
\begin{thebibliography}{10}

\bibitem{basic2015}
N.~Ba\v{s}i\'{c} and T.~Pisanski, \emph{Iterated altans and their properties},
  MATCH Commun. Math. Comput. Chem. \textbf{74} (2015), no.~3, 653--666.

\bibitem{brinkmann2005}
G.~Brinkmann, O.~Delgado-Friedrichs, and U.~von Nathusius, \emph{Numbers of
  faces and boundary encodings of patches}, Graphs and {D}iscovery, DIMACS Ser.
  Discrete Math. Theoret. Comput. Sci., vol.~69, Amer. Math. Soc., 2005,
  pp.~27--38.

\bibitem{brinkmann2009}
G.~Brinkmann, J.~E. Graver, and C.~Justus, \emph{Numbers of faces in disordered
  patches}, J. Math. Chem. \textbf{45} (2009), no.~2, 263--278.

\bibitem{cruzgutmanrada2012}
R.~Cruz, I.~Gutman, and J.~Rada, \emph{Convex hexagonal systems and their
  topological indices}, MATCH Commun. Math. Comput. Chem. \textbf{68} (2012),
  no.~1, 97--108.

\bibitem{cyvin1991}
S.~J. Cyvin, J.~Brunvoll, and B.~N. Cyvin, \emph{{T}heory of {C}oronoid
  {H}ydrocarbons}, Lecture Notes in Chemistry, vol.~54, Springer-Verlag, 1991.

\bibitem{cyvin1988}
S.~J. Cyvin and I.~Gutman, \emph{{K}ekulé {S}tructures in {B}enzenoid
  {H}ydrocarbons}, Lecture Notes in Chemistry, vol.~48, Springer-Verlag, 1988.

\bibitem{dickensmallion2014a}
T.~K. Dickens and R.~B. Mallion, \emph{{$\pi$}-{E}lectron ring-currents and
  bond-currents in some conjugated altan-structures}, J. Phys. Chem. A
  \textbf{118} (2014), no.~20, 3688--3697.

\bibitem{dickensmallion2014}
T.~K. Dickens and R.~B. Mallion, \emph{Topological {H}{\"
  u}ckel-{L}ondon-{P}ople-{M}c{W}eeny ring currents and bond currents in
  altan-corannulene and altan-coronene}, J. Phys. Chem. A \textbf{118} (2014),
  no.~5, 933--939.

\bibitem{fowlermyrvold2011}
P.~W. Fowler and W.~Myrvold, \emph{The ``anthracene problem'': Closed-form
  conjugated-circuit models of ring currents in linear polyacenes}, J. Phys.
  Chem. A \textbf{115} (2011), no.~45, 13191--13200.

\bibitem{gomes1979}
J.~A. N.~F. Gomes and R.~B. Mallion, \emph{A quasi-topological method for the
  calculation of relative `ring current' intensities in polycyclic, conjugated
  hydrocarbons}, Rev. Port. Quim. \textbf{21} (1979), 82--89.

\bibitem{graver2003}
J.~E. Graver, \emph{{T}he $(m, k)$-patch {B}oundary {C}ode {P}roblem}, MATCH
  Commun. Math. Comput. Chem. \textbf{48} (2003), 189--196.

\bibitem{gravergraves2010}
J.~E. Graver and C.~M. Graves, \emph{Fullerene patches {I}}, Ars Math. Contemp.
  \textbf{3} (2010), no.~1, 109--120.

\bibitem{gravergravesgraves2014}
J.~E. Graver, C.~M. Graves, and S.~J. Graves, \emph{Fullerene patches {II}},
  Ars Math. Contemp. \textbf{7} (2014), no.~2, 405--421.

\bibitem{gravesgraves2014}
C.~Graves and S.~J. Graves, \emph{Counting symmetric and near-symmetric
  fullerene patches}, J. Math. Chem. \textbf{52} (2014), no.~9, 2423--2441.

\bibitem{graves2012}
C.~Graves and J.~McLoud-Mann, \emph{Side lengths of pseudoconvex fullerene
  patches}, Ars Math. Contemp. \textbf{5} (2012), no.~2, 291--302.

\bibitem{extendingpatches2015}
C.~M. Graves, J.~McLoud-Mann, and K.~Stagg~Rovira, \emph{Extending patches to
  fullerenes}, Ars Math. Contemp. \textbf{9} (2015), no.~2, 219--232.

\bibitem{grunbaum2013tilings}
B.~Gr{\" u}nbaum and G.~C. Shephard, \emph{{T}ilings and {P}atterns}, Dover
  Books on Mathematics Series, Dover Publications, Incorporated, 2013.

\bibitem{guo2002}
X.~Guo, P.~Hansen, and M.~Zheng, \emph{Boundary uniqueness of fusenes}, Disc.
  Appl. Math. \textbf{118} (2002), no.~3, 209--222.

\bibitem{gutman2014a}
I.~Gutman, \emph{Altan derivatives of a graph}, Iranian J. Math. Chem.
  \textbf{5} (2014), 85--90.

\bibitem{gutman2014}
I.~Gutman, \emph{Topological properties of altan-benzenoid hydrocarbons}, Serb.
  Chem. Soc. \textbf{79} (2014), no.~12, 1515--1521.

\bibitem{gutman1989}
I.~Gutman and S.~J. Cyvin, \emph{{I}ntroduction to the {T}heory of {B}enzenoid
  {H}ydrocarbons}, Springer-Verlag, 1989.

\bibitem{hammack2011handbook}
R.~Hammack, W.~Imrich, and S.~Klav{\v{z}}ar, \emph{{H}andbook of {P}roduct
  {G}raphs}, CRC press, 2011.

\bibitem{hansen1996}
P.~Hansen, C.~Lebatteux, and M.~Zheng, \emph{The boundary-edges code for
  polyhexes}, J. Mol. Struct. (Theochem) \textbf{363} (1996), no.~2, 237--247.

\bibitem{hatcher2002algebraic}
A.~Hatcher, \emph{{A}lgebraic {T}opology}, Cambridge University Press, 2002.

\bibitem{herndon1987}
W.~C. Herndon and A.~J. Bruce, \emph{Perimeter code for benzenoid aromatic
  hydrocarbons}, Graph Theory and Topology in Chemistry (D.~H. King, R. B.
  \and~Rouvray, ed.), Studies in Physical and Theoretical Chemistry, vol.~51,
  Elsevier Science Ltd, 1987, pp.~491--513.

\bibitem{kovic2014}
J.~Kovič, T.~Pisanski, A.~T. Balaban, and P.~W. Fowler, \emph{On symmetries of
  benzenoid systems}, MATCH Commun. Math. Comput. Chem. \textbf{72} (2014),
  no.~1, 3--26.

\bibitem{lukovits2003}
I.~Lukovits, A.~Graovac, E.~Kalman, G.~Kaptay, P.~Nagy, S.~Nikoli\'{c},
  J.~Sytchev, and N.~Trinajsti\'{c}, \emph{Nanotubes: Number of {K}ekulé
  structures and aromaticity}, J. Chem. Inf. Comput. Sci. \textbf{43} (2003),
  no.~2, 609--614.

\bibitem{monacomemoli2013}
G.~Monaco, M.~Memoli, and R.~Zanasi, \emph{Additivity of current density
  patterns in altan-molecules}, J. Phys. Org. Chem. \textbf{26} (2013), no.~2,
  109--114.

\bibitem{monacozanasi2009}
G.~Monaco and R.~Zanasi, \emph{On the additivity of current density in
  polycyclic aromatic hydrocarbons}, J. Chem. Phys. \textbf{131} (2009), no.~4,
  044126.

\bibitem{monaco2012}
G.~Monaco and R.~Zanasi, \emph{Three contra-rotating currents from a rational
  design of polycyclic aromatic hydrocarbons: altan-corannulene and
  altan-coronene}, J. Phys. Chem. A \textbf{116} (2012), no.~36, 9020--9026.

\bibitem{monaco2013}
G.~Monaco and R.~Zanasi, \emph{Anionic derivatives of altan-corannulene}, J.
  Phys. Org. Chem. \textbf{26} (2013), 730--736.

\bibitem{putz2011}
M.~V. Putz, \emph{Carbon bonding and structures: Advances in physics and
  chemistry}, Carbon Materials: Chemistry and Physics, Springer Netherlands,
  2011.

\bibitem{randic2010}
M.~Randi\'{c}, \emph{Graph theoretical approach to {$\pi$}-electron currents in
  polycyclic conjugated hydrocarbons}, Chem. Phys. Lett. \textbf{500} (2010),
  no.~1--3, 123--127.

\bibitem{sachs1996}
H.~Sachs, P.~Hansen, and M.~L. Zheng, \emph{Kekulé count in tubular
  hydrocarbons}, MATCH Commun. Math. Comput. Chem. (1996), no.~33, 169--241.

\bibitem{thomassen1992}
C.~Thomassen, \emph{The {J}ordan-{S}ch{\"o}nflies theorem and the
  classification of surfaces}, Amer. Math. Monthly \textbf{99} (1992), no.~2,
  116--130.

\end{thebibliography}

\end{document}